\documentclass[12pt]{article}

\usepackage{amsmath}
\usepackage{amssymb}
\usepackage{amsthm}
\usepackage{hyperref}

\newtheorem{thm}{Theorem}[section]

\newtheorem{lem}[thm]{Lemma}

\newtheorem{assumption}[thm]{Assumption}

\newtheorem{definition}[thm]{Definition}

\newtheorem{example}[thm]{Example}
\newenvironment{exmp}{\begin{example}\rm}{\end{example}}
\newtheorem{remark}[thm]{Remark}

\newtheorem{tab}{Table}

\newcommand{\eps}{\varepsilon}

\title{Analyticity of Entropy Rate of Continuous-State Hidden Markov Chains}

\author{\begin{tabular}{cc}
Guangyue Han&Brian Marcus\\
University of Hong Kong&University of British Columbia\\
{\em email:} ghan@hku.hk&{\em email:} marcus@math.ubc.ca\\
\end{tabular}}

\date{{\normalsize \today}}

\begin{document}\maketitle\thispagestyle{empty}

\begin{abstract}
We prove that under certain mild assumptions, the entropy rate of a
hidden Markov chain, observed when passing a finite-state stationary
Markov chain through a discrete-time continuous-output channel, is
jointly analytic as a function of the input Markov chain parameters
and the channel parameters. In particular, as consequences of the
main theorems, we obtain analyticity for the entropy rate associated
with representative channels: Cauchy and Gaussian.
\end{abstract}

\section{Main Results and Related Work} \label{I}

Entropy rate for hidden Markov chains is notoriously difficult to
compute, even in the case where both input and output alphabets are
finite and time is discrete.  However, recently much progress has
been made in this setting; see, for instance~\cite{hm09,jss08,or11,zu05}
and the references therein.

In this paper,  we consider a discrete-time channel with a finite input alphabet $\mathcal{Y}=\{1, 2, \cdots, l\} \subset \mathbb{R}$ and the continuous output alphabet $\mathcal{Z}=\mathbb{R}$. Here, we remark that all our results in this paper can be straightforwardly translated to the setting where $\mathcal{Z}$ is finite or countably infinite.

We assume that the input process is a $\mathcal{Y}$-valued
first-order stationary Markov chain $Y$ with transition probability
matrix $\Pi=(\pi_{i j})_{l \times l}$ and stationary vector
$\pi=(\pi_{i})_{1 \times l}$ (here we assume $Y$ is first-order only
for simplicity; a standard ``blocking'' approach can be used to
reduce higher order cases to the first-order case).

We assume that the channel is memoryless in the sense that at each
time, the distribution of the output $z \in \mathcal{Z}$, given the
input $y \in \mathcal{Y}$, is independent of the past and future
inputs and outputs, and is distributed according to a probability
density function $q(z|y)$.

The corresponding output process of this channel is a {\em hidden Markov chain},
which will be denoted by $Z$ throughout the paper. The entropy rate $H(Z)$ is defined as
$$
H(Z)=\lim_{n \to \infty} \frac{1}{n+1} H(Z_{-n}^0),
$$
when the limit exists, where
$$
H(Z_{-n}^0)=-\int_{\mathcal{Z}^{n+1}} p(z_{-n}^0) \log p(z_{-n}^{0})
dz_{-n}^0;
$$
here $z_{-n}^0 \triangleq (z_{-n}, z_{-n+1}, \cdots, z_0)$  denotes
an instance of $Z_{-n}^0 \triangleq (Z_{-n}, Z_{-n+1}, \cdots,
Z_0)$, and $p(z_{-n}^0)$ denotes the probability density of
$z_{-n}^0$. It is well-known (e.g., see page $60$
of~\cite{continuous-system}) that if $H(Z_{-n}^0)$ is finite for all
$n$, then the limit above exists and can be written as
$$
H(Z)=\lim_{n \to \infty} H_n(Z),
$$
where
\begin{equation}  \label{n-th}
H_n(Z)=-\int_{\mathcal{Z}^{n+1}} p(z_{-n}^0) \log p(z_0|z_{-n}^{-1})
dz_{-n}^0;
\end{equation}
here $p(z_0|z_{-n}^{-1})$ denotes the conditional density of $z_0$
given $z_{-n}^{-1}$. Since  the channels considered in this paper
are memoryless and $Y$ is stationary, we have
$$
H(Z_{-n}^0|Y_{-n}^0)=(n+1) H(Z_0|Y_0),
$$
where
$$
H(Z_0|Y_0) = -\sum_{y \in \mathcal{Y}} \pi_y \int_{\mathcal{Z}}
q(z|y) \log q(z|y) dz.
$$
It then follows from
$$
H(Z_{-n}^0|Y_{-n}^0) \leq H(Z_{-n}^0) \leq H(Y_{-n}^0) + H(Z_{-n}^0|Y_{-n}^0)
$$
that if $ \int_{z \in \mathcal{Z}} q(z|y) \log q(z|y) dz $ is finite
for all $y \in \mathcal{Y}$, then the formulas for $H(Z)$ above hold
and $H(Z)$ is finite.

Unless specified otherwise, we will assume that $\Pi =
\Pi^{\vec{\eps}} = (\pi^{\vec{\eps}}_{ij})$ is analytically
parameterized by $\vec{\eps} \in \Omega_1$, where $\Omega_1$ denotes a bounded domain in
$\mathbb{R}^{m_1}$ (here, a domain is an open and connected set),
and for any $(y, z) \in \mathcal{Y} \times \mathcal{Z}$,
$q^{\vec{\theta}}(z|y)$ is analytically parameterized by
$\vec{\theta} \in \Omega_2$, where $\Omega_2$ denotes a bounded
domain in $\mathbb{R}^{m_2}$.

In earlier work, we established analyticity of entropy rate as a
function of the underlying Markov chain parameters:
\begin{thm} \label{old} ~\cite[Theorem 1.1]{hm10} Assume that for any $y \in
\mathcal{Y}$, $q(z|y)$ is positive and continuous on $\mathcal{Z}$,
and the following two integrals
\begin{equation}\label{forgot} \int_{\mathcal{Z}}
q(z|y) |\log \min_{y'} q(z|y')| dz, \qquad \int_{\mathcal{Z}} q(z|y)
|\log \max_{y'} q(z|y')| dz
\end{equation}
are finite. If $\Pi$ is strictly positive at $\vec{\eps}_0$, then
$H(Z)$ is analytic around $\vec{\eps}_0$.
\end{thm}
\noindent (we remark that the hypotheses (\ref{forgot}) were inadvertently omitted in~\cite[Theorem 1.1]{hm10}.)

The aim of the current paper is to prove analyticity as a function
of both the Markov chain parameters and channel parameters. Simple
examples show that $H(Z)$ can fail to be analytic as a function of
the channel parameter alone; see Example~\ref{example}. Our positive
results require several technical  regularity conditions, which we
describe as follows.  These conditions involve the complexification
of the channel density functions
(by definition, any real analytic function, such as $q^{\vec{\theta}}(z|y)$ as above, at a given point can be
uniquely extended to a complex analytic function on some complex
neighborhood of the given point; we will continue to use the
same notation, such as $q^{\vec{\theta}}(z|y)$, for this complex extension).
We require these technical conditions, which abstract the properties
of commonly used probability density functions (e.g., Cauchy and
Gaussian), in order to make our proofs work. There may be more
general conditions that suffice. On a first reading, the reader may
want to skip directly to the statements of results below.

Our regularity conditions are as follows.  For given $(\vec{\eps}_0,
\vec{\theta}_0) \in \Omega_1 \times \Omega_2$,
\begin{enumerate}
\item[(a)] $\Pi$ is strictly positive at $\vec{\eps}_0$;
\item[(b)] for all $y \in \mathcal{Y}$,  $q^{\vec{\theta}_0}(z|y)$ is positive on $\mathcal{Z}$;
\item[(c)] there exists $r_2 > 0$ such that (below, $\mathbb{C}_{\vec{\theta}_0}^{m_2}(r_2)$ denotes the $r_2$-neighborhood of $\vec{\theta}_0$ in $\mathbb{C}^{m_2}$)
\begin{itemize}
\item for any $(y, z) \in \mathcal{Y} \times \mathcal{Z}$,
$q^{\vec{\theta}}(z|y)$ is analytic on $\mathbb{C}_{\vec{\theta}_0}^{m_2}(r_2)$,
\item for any $y \in \mathcal{Y}$,
 $q^{\vec{\theta}}(z|y)$ is
jointly continuous on $\mathcal{Z}\times
\mathbb{C}_{\vec{\theta}_0}^{m_2}(r_2)$,
\item the following three integrals
\begin{equation} \label{max-of-three}
\hspace{-2cm} (i) ~~ \int \breve{q}(z, r_2) dz, ~~~ (ii) ~~ \int \breve{q}(z, r_2)
\log \hat{q}(z, r_2) dz, ~~~ (iii)~~ \int \breve{q}(z, r_2) \log
\breve{q}(z, r_2) dz,
\end{equation}
are all finite, where
$$
\breve{q}(z, r_2) \triangleq \sup_{(y, \vec{\theta}) \in \mathcal{Y} \times \mathbb{C}_{\vec{\theta}_0}^{m_2}(r_2)} |q^{\vec{\theta}}(z|y)|, \qquad \hat{q}(z, r_2) \triangleq \inf_{(y, \vec{\theta}) \in \mathcal{Y} \times \mathbb{C}_{\vec{\theta}_0}^{m_2}(r_2)} |q^{\vec{\theta}}(z|y)|;
$$
\end{itemize}

\item[(d)] for some $I \in \{1, 2, \ldots, l\}$,
\begin{itemize}
\item[(i)] there exist $r_2 > 0$ such that for all $j$, the family of
functions $\{g_z(\theta) = q^{\vec{\theta}}(z|j)/q^{\vec{\theta}}(z|I)\}_z$ on $\vec{\theta} \in
\mathbb{C}_{\vec{\theta}_0}^{m_2}(r_2)$ is equicontinuous,

\item[(ii)] there exists $r_2 > 0$ such that for each $z$, the real $\log q^{\vec{\theta}}(z|I)$ can be analytically extended to $\mathbb{C}_{\vec{\theta}_0}^{m_2}(r_2)$ and for all $j$,
\begin{equation} \label{first-term-analytic-1}
\int_{\mathcal{Z}} ~~ \sup_{\vec{\theta} \in \mathbb{C}_{\vec{\theta}_0}^{m_2}(r_2)} \left| q^{\vec{\theta}}(z|j) \log q^{\vec{\theta}}(z|I) \right|
dz < \infty.
\end{equation}
\end{itemize}
\end{enumerate}
\noindent 
It is easily seen that the most commonly used channel models,
including Cauchy and Gaussian channels, satisfy all of
these conditions (these channels are defined below).

Our first result deals with the case where the
$q^{\vec{\theta}}(z|i)$ are in some sense``comparable'' with one
another.

\begin{thm} \label{Cauchy}
For given $(\vec{\eps}_0, \vec{\theta}_0) \in \Omega_1 \times
\Omega_2$, assume Conditions (a), (b), (c) and (d). If, there exist
$C', C'' > 0$ such that for all $i, j$ and all $z \in \mathcal{Z}$,
\begin{equation} \label{all-columns-comparable}
C' \leq \left|
\frac{q^{\vec{\theta}_0}(z|j)}{q^{\vec{\theta}_0}(z|i)} \right| \leq
C'',
\end{equation}
then $H(Z)$ is analytic around $(\vec{\eps}_0, \vec{\theta}_0)$.
\end{thm}
\noindent Theorem~\ref{Cauchy} applies to the additive Cauchy
channel, parameterized by $(\gamma_i, \mu_i)$, $\gamma_i > 0$, with
\begin{equation}  \label{q-z-y-Cauchy}
q(z|i)=\frac{1}{\pi} \frac{\gamma_i}{(z-\mu_i)^2+\gamma_i^2}.
\end{equation}
So, for this channel,  if $\Pi$ is strictly positive  at $\vec{\eps}_0 \in \Omega_1$, then $H(Z)$ is analytic around
$(\vec{\eps}_0, (\gamma_{1}, \mu_1, \gamma_{2}, \mu_2, \ldots,
\gamma_{l}, \mu_{l}))$.

Our next result deals  with the case where for one particular $i$,
the density $q^{\vec{\theta}}(z|i)$ dominates all the others.
The reader should note that while Theorem~\ref{Cauchy} above
requires a condition to hold  at one given parameter value,
$\vec{\theta}_0$, Theorem~\ref{Gaussian} below requires a condition
to hold for all parameter values $\vec{\theta}$ in a complex
neighborhood of the given $\vec{\theta}_0$.

\begin{thm} \label{Gaussian}
For given $(\vec{\eps}_0, \vec{\theta}_0) \in \Omega_1 \times
\Omega_2$, assume  Conditions (a), (b), (c) and (d). If, in
addition, for the same $I$ as in Condition (d), there exists $r_2 >
0$ such that for any $\eps > 0$, there exists a compact subset
$\Sigma \subset \mathcal{Z}$ such that for all $z \not \in \Sigma$,
all $j \neq I$ and all $\vec{\theta} \in
\mathbb{C}_{\vec{\theta}_0}^{m_2}(r_2)$
\begin{equation} \label{one-column-dominant}
\left| \frac{q^{\vec{\theta}}(z|j)}{q^{\vec{\theta}}(z|I)} \right|
\leq \eps,
\end{equation}
then $H(Z)$ is analytic around $(\vec{\eps}_0, \vec{\theta}_0)$.
\end{thm}
\noindent Theorem~\ref{Gaussian} applies to the additive Gaussian
channel parameterized by $(\sigma_i, \mu_i)$, $\sigma_i
> 0$, with
\begin{equation}  \label{q-z-y-Gaussian}
q(z|i)=\frac{1}{\sqrt{2\pi}\sigma_i} e^{-(z-\mu_i)^2/(2\sigma_i^2)},
\end{equation}
where some $\sigma_i$ is strictly larger than all other ones. Setting $I = i$ to be the index
corresponding to the largest value of $\sigma_i$, it is easy to see
that indeed (\ref{q-z-y-Gaussian}) satisfies Condition (\ref{one-column-dominant}). So, for this channel,  if $\Pi$ is strictly positive  at $\vec{\eps}_0 \in \Omega_1$, then $H(Z)$ is analytic around $(\vec{\eps}_0,
(\sigma_{1}, \mu_1, \sigma_{2}, \mu_2, \cdots, \sigma_{l}, \mu_{l}))$.

%
%
Our final result deals with a case of more theoretical interest,
namely:  for all $i$, the real part of $q^{\theta}(z|i)$
``dominates'' the imaginary part of the complex extension.

\begin{thm} \label{main}
For given $(\vec{\eps}_0, \vec{\theta}_0) \in \Omega_1 \times
\Omega_2$, assume Conditions (a), (b) and (c). If, in addition, for
any $\delta > 0$,  there exists $r_2 > 0$ such that for all $(y, z)
\in (\mathcal{Y}, \mathcal{Z})$ and all $\vec{\theta} \in
\mathbb{C}_{\vec{\theta}_0}^{m_2}(r_2)$
\begin{equation} \label{dominated-by-real}
(i)~~ |\Im(q^{\vec{\theta}}(z|y))| < \delta |\Re(q^{\vec{\theta}}(z|y))|,
\quad (ii)~~\left| \log
\frac{q^{\vec{\theta}}(z|y)}{q^{\vec{\theta}_0}(z|y)} \right| \leq
\delta,
\end{equation}
then $H(Z)$ is analytic around $(\vec{\eps}_0, \vec{\theta}_0)$.
\end{thm}


Theorem~\ref{main} applies to additive Cauchy channels parameterized
in (\ref{q-z-y-Cauchy}) and other more artificial channels.

These results can be regarded as extensions of~\cite[Theorem
1.1]{gm05}, which deals with the case where  $\mathcal{Z}$
is finite. The flow of the proofs of these results
follows that of this case. However, new techniques are
needed to deal with the continuous case.  This is most notable in
Theorem~\ref{main}, where the use of a complex Hilbert
metric~\cite{hm09a}, to replace the
classical real Hilbert metric, is critical.  This metric was also
used in the proof of Theorem~\ref{old}. It is not needed for
Theorems~\ref{Cauchy} and~\ref{Gaussian}, because those results
assume stronger conditions on the channel density functions.

We remark that in~\cite[Theorem 1.1]{gm05}, zero values are allowed
for some transition probabilities.  It seems more difficult to
handle this phenomena in the continuous-state setting; this is the
subject of forthcoming work.

To the best of our knowledge,  the results in this paper, together
with those in~\cite{hm10}, are among the first results establishing
analyticity of continuous-state hidden Markov chains. Given the
interest in the counterpart results for the discrete-state setting,
we expect that such results will be of significance in the
continuous-state setting as well.

The remainder of this paper is organized as follows. In
Section~\ref{main-idea}, we review the (real) Hilbert metric,
outline the framework of the proofs of our theorems and highlight
the differences among the proofs. The following two sections are
devoted to proving Theorems~\ref{Cauchy} and~\ref{Gaussian}. In
Section~\ref{Hilbert}, we review the complex Hilbert metric. In
Section~\ref{III}, we prove Theorem~\ref{main}.

\section{The Main Idea of the Proofs} \label{main-idea}

We first briefly review the classical (real) Hilbert metric. The
real Hilbert metric will be used in the proofs of
Theorems~\ref{Cauchy} and~\ref{Gaussian}.
%

Let $W$ be the standard simplex in the $l$-dimensional real
Euclidean space,
$$
W = \{w=(w_1, w_2, \cdots, w_{l}) \in \mathbb{R}^{l}:w_i \geq 0,
\sum_i w_i=1\},
$$
and let $W^{\circ}$ denote its interior, consisting of the vectors
with positive coordinates. For any two vectors $v, w \in W^{\circ}$,
the Hilbert metric~\cite{se80} is defined as
\begin{equation}  \label{HilbertMetric}
d_H(w, v)=\max_{i, j} \log \left( \frac{w_i/w_j}{v_i/v_j} \right).
\end{equation}
It is well known  and easy to see that on $W^\circ$, the Hilbert
metric dominates the Euclidean metric up to a positive constant
(i.e., for some $K > 0$ and all $x,y \in W$, $|x - y| \le K
d_H(x,y)$); also, on any compact subset of $W^{\circ}$, the two
metrics are equivalent (see Proposition $2.1$ of~\cite{gm05}).

For an $l \times l$ positive matrix $T=(t_{ij})$ (i.e., each $t_{ij}
> 0$), the mapping $f_T$ induced by $T$ on $W$ is defined by
\begin{equation} \label{Induce}
f_T(w) = \frac{wT}{w T{\bf 1}},
\end{equation}
where $\bf 1$ is the all $1$'s column vector. The following theorem
is well-known (see~\cite{se80}).
\begin{thm} \label{real-contraction}
For a positive $T$, $f_T$ is a contraction mapping on the entire $W^{\circ}$
under the Hilbert metric and the contraction coefficient (often
referred to as the Birkhoff coefficient), is given by
\begin{equation}  \label{BirkhoffCoefficient}
\tau(T)=\sup_{v \neq w} \frac{d_H(vT, wT)}{d_H(v,
w)}=\frac{1-\sqrt{\phi(T)}}{1+\sqrt{\phi(T)}},
\end{equation}
where $\phi(T)=\min_{i, j, k, l} \frac{t_{ik}t_{jl}}{t_{jk}t_{il}}$.
\end{thm}

We will also need a complex version of $W$,
$$
\tilde{W}=\{w=(w_1, w_2, \cdots, w_{l}) \in \mathbb{C}^{l} : \sum_i
w_i=1\}.
$$
And for any $D \subset W$ and $\delta > 0$, we define
$$
\tilde{W}_D(\delta) = \{\tilde{w} \in \tilde{W}: |\tilde{w}-w| < \delta \mbox{ for some } \tilde{w}\}.
$$

For each $z \in \mathcal{Z}$, define $\Pi^{\vec{\eps}, \vec{\theta}}(z)$ as an $l \times l$ matrix with the entries
\begin{equation} \label{Piz}
\Pi^{\vec{\eps}, \vec{\theta}}(z)_{ij}=\pi_{ij}^{\vec{\eps}}
q^{\vec{\theta}}(z|j), \mbox{ for all } i, j.
\end{equation}
By (\ref{Induce}), $\Pi^{\vec{\eps}, \vec{\theta}}(z)$ will induce a mapping $f_z^{\vec{\eps}, \vec{\theta}} \triangleq f_{\Pi^{\vec{\eps}, \vec{\theta}}(z)}$ from $W$ to $W$. For any fixed $n$ and $z_{-n}^0$, define
\begin{equation} \label{x-i}
x_i^{\vec{\eps}, \vec{\theta}}=x_i^{\vec{\eps}, \vec{\theta}}(z_{-n}^i)=p^{\vec{\eps}, \vec{\theta}}(y_i=\cdot \;|z_i, z_{i-1}, \cdots, z_{-n}),
\end{equation}
(here $\cdot$ represent the states of the Markov chain $Y$) then
similar to Blackwell~\cite{bl57}, $\{x_i^{\vec{\eps}, \vec{\theta}}\}$ satisfies the random dynamical
system
\begin{equation} \label{iter0}
x_{i+1}^{\vec{\eps}, \vec{\theta}}=f_{z_{i+1}}^{\vec{\eps}, \vec{\theta}}(x_i^{\vec{\eps}, \vec{\theta}}),
\end{equation}
starting with
\begin{equation} \label{starting}
x_{-n-1}^{\vec{\eps}, \vec{\theta}} = \pi^{\vec{\eps}},
\end{equation}
where $\pi^{\vec{\eps}}$ is the stationary vector for
$\Pi^{\vec{\eps}}$.
 And obviously we have
\begin{equation} \label{amalg}
p^{\vec{\eps}, \vec{\theta}}(z_0|z_{-n})=x_{-1}^{\vec{\eps},
\vec{\theta}} \Pi^{\vec{\eps}, \vec{\theta}}(z_0) \mathbf{1},
\end{equation}
and
\begin{equation} \label{joint}
p^{\vec{\eps}, \vec{\theta}}(z_{-n}^0)=\pi^{\vec{\eps}}
\Pi^{\vec{\eps}, \vec{\theta}}(z_{-n}) \Pi^{\vec{\eps},
\vec{\theta}}(z_{-n+1}) \cdots \Pi^{\vec{\eps}, \vec{\theta}}(z_0)
\mathbf{1}.
\end{equation}

Evidently $x_i^{\vec{\eps}, \vec{\theta}}$, $p^{\vec{\eps},
\vec{\theta}}(z_0|z_{-n})$ and $p^{\vec{\eps},
\vec{\theta}}(z_{-n}^0)$ all depend on the real vector $(\vec{\eps},
\vec{\theta}) \in \Omega_1 \times \Omega_2$. In what follows, we
shall show that they can be ``complexified''. For $r_1, r_2 > 0$,
let $\mathbb{C}_{\vec{\eps}_0}^{m_1}(r_1)$ denote a $r_1$-ball
around $\vec{\eps}_0$ in $\mathbb{C}^{m_1}$, and
$\mathbb{C}_{\vec{\theta}_0}^{m_2}(r_2)$ denote a $r_2$-ball around
$\vec{\theta}_0$ in $\mathbb{C}^{m_2}$. For any $\vec{\eps} \in
\mathbb{C}_{\vec{\eps}_0}^{m_1}(r_1)$, one checks that for $r_1 > 0$
small enough, the stationary vector $\pi^{\vec{\eps}}$ is unique and
analytic on $\mathbb{C}_{\vec{\eps}_0}^{m_1}(r_1)$ as a function of
$\vec{\eps}$ (because it is the unique solution of
$$
\pi^{\vec{\eps}} \Pi^{\vec{\eps}}=\pi^{\vec{\eps}}, \qquad \sum_y
\pi_y^{\vec{\eps}}=1 ).
$$

Then through (\ref{starting}) and (\ref{iter0}), $x_i^{\vec{\eps},
\vec{\theta}}$ can be analytically extended to
$\mathbb{C}_{\vec{\eps}_0}^{m_1}(r_1) \times
\mathbb{C}_{\vec{\theta}_0}^{m_2}(r_2)$; furthermore, through
(\ref{amalg}) and (\ref{joint}), $p^{\vec{\eps},
\vec{\theta}}(z_0|z_{-n})$ and $p^{\vec{\eps},
\vec{\theta}}(z_{-n}^0)$ can be analytically extended to
$\mathbb{C}_{\vec{\eps}_0}^{m_1}(r_1) \times
\mathbb{C}_{\vec{\theta}_0}^{m_2}(r_2)$. Ultimately
$H_n^{\vec{\eps}, \vec{\theta}}(Z)$ (which is defined by
(\ref{n-th}) with real superscripts $(\vec{\eps}, \vec{\theta})$ on
$p(z_{-n}^0)$ and $p(z_0|z_{-n}^{-1})$) can be analytically extended
to $\mathbb{C}_{\vec{\eps}_0}^{m_1}(r_1) \times
\mathbb{C}_{\vec{\theta}_0}^{m_2}(r_2)$ as well.

The proofs of the main results require the mappings $f^{\vec{\eps},
\vec{\theta}}_z$ to be contraction mappings with respect to the
Euclidean metric.  This will be derived from contraction, with
respect to the Hilbert metric, the equivalence of the Hilbert and
the Euclidean metrics and re-blocking into non-overlapping blocks.
Namely, we will consecutively re-block the $Z$ process to a
$\hat{Z}$ process such that $\hat{Z}_i$ is of the form $Z_{j(i)}^{k(i)}$.
For any $\hat{z}_i=z_{j(i)}^{k(i)}$, let $f^{\vec{\eps},
\vec{\theta}}_{\hat{z}_i}$ denote the composed mapping
$f^{\vec{\eps}, \vec{\theta}}_{z_k} \circ f^{\vec{\eps},
\vec{\theta}}_{z_{k-1}} \circ \cdots \circ f^{\vec{\eps},
\vec{\theta}}_{z_j}$. Then, $\hat{x}_i$ can be similarly defined as
before: For any fixed $\hat{n}$ and $z_{-\hat{n}}^0$, define
\begin{equation} \label{hat-x-i}
\hat{x}_i^{\vec{\eps}, \vec{\theta}}=\hat{x}_i^{\vec{\eps},
\vec{\theta}}(\hat{z}_{-\hat{n}}^i)=p^{\vec{\eps},
\vec{\theta}}(y_{k(i)}=\cdot \;|\hat{z}_i, \hat{z}_{i-1}, \cdots,
\hat{z}_{-\hat{n}}),
\end{equation}
then, $\{\hat{x}_i^{\vec{\eps}, \vec{\theta}}\}$ satisfies the random dynamical
system
\begin{equation} \label{hat-iter0}
\hat{x}_{i+1}^{\vec{\eps}, \vec{\theta}}=f_{\hat{z}_{i+1}}^{\vec{\eps}, \vec{\theta}}(\hat{x}_i^{\vec{\eps}, \vec{\theta}}),
\end{equation}
starting with
\begin{equation} \label{hat-starting}
\hat{x}_{-\hat{n}-1}^{\vec{\eps}, \vec{\theta}} = \pi^{\vec{\eps}}.
\end{equation}

The framework for the proofs of the main theorems can be outlined as follows:
\begin{enumerate}
\item[(I)] If necessary, we consecutively re-block the $Z$ process to a $\hat{Z}$ process such that $\hat{Z}_i$ is of the form $Z_{j(i)}^{k(i)}$.
\item[(II)] We then show that there exists a complex neighborhood of a subset of $W$ such that each complexified $f_{\hat{z}}^{\vec{\eps}, \vec{\theta}}$ is a
contraction mapping, with respect to some metric, and moreover the
complexified $\hat{x}_i^{\vec{\eps},
\vec{\theta}}(\hat{z}_{-\hat{n}}^i)$ stays within the neighborhood.
\item[(III)] It then follows that the complexified $x_i^{\vec{\eps}, \vec{\theta}}(z_{-n}^i)$ and thus the complexified $p^{\vec{\eps}, \vec{\theta}}(z_0|z_{-n})$ exponentially forget their initial conditions.
\item[(IV)] This, together with bounding arguments, will further imply that the complexified  $H_n^{\vec{\eps}, \vec{\theta}}(Z)$ uniformly converges to a complex analytic function, which is
necessarily the complexified $H^{\vec{\eps}, \vec{\theta}}(Z)$, on a complex domain, and therefore $H^{\vec{\eps}, \vec{\theta}}(Z)$ is analytic.
\end{enumerate}
Although the three proofs all fit in the same above-mentioned
framework, there does not seem to be a natural way to unify them.
Among numerous differences, the most essential one is the way we
establish (II):
\begin{itemize}
\item For Theorem~\ref{Cauchy}, as $i$ increases, the
real $\hat{x}_i^{\vec{\eps}, \vec{\theta}}$ always stays within a
compact subset $D$ of $W^\circ$. To establish (II), we will use the
fact that each real $f_z^{\vec{\eps}, \vec{\theta}}$ is a
contraction on $W^\circ$, and thus on $D$, with respect to the
Hilbert metric. Then we use the equivalence between the Euclidean
metric and the Hilbert metric on $D$, and equicontinuity in Condition (d(i)) to
establish the contractiveness (with respect to the Euclidean metric)
of the complexified $f_{\hat{z}}^{\vec{\eps}, \vec{\theta}}$ on a
complex neighborhood of $D$.

\item For Theorem~\ref{Gaussian}, as $i$ increases, the real $\hat{x}_i^{\vec{\eps}, \vec{\theta}}$
may move arbitrarily close to the boundary of $W$. To establish
(II), we apply a dichotomy argument: there is a sufficiently large
compact subset $\Sigma$ of $\mathcal{Z}$ such that for any
$\hat{z}_i=z_{j(i)}^{k(i)}$ with $z_{j(i)} \in \Sigma$, the
complexified $f_{\hat{z}}^{\vec{\eps}, \vec{\theta}}$ is a
contraction on a complex neighborhood (under the Euclidean metric)
of $W^\circ$; for any $\hat{z}_i=z_{j(i)}^{k(i)}$ with $z_{j(i)} \not \in \Sigma$,
we directly establish (II) by estimating the first-order derivative
of the complexified $f_{\hat{z}}^{\vec{\eps}, \vec{\theta}}$.

\item For Theorem~\ref{main}, as $i$ increases, the
real $x_i^{\vec{\eps}, \vec{\theta}}$ may move arbitrarily close to
the boundary of $W$. To establish (II), the complex Hilbert metric
in Section~\ref{Hilbert} is employed to directly show the
contractiveness, with respect to the complex Hilbert metric, of the
complexified $f_{z}^{\vec{\eps}, \vec{\theta}}$ on a complex
neighborhood of $W^\circ$.
\end{itemize}

\section{Proof of Theorem~\ref{Cauchy}}

The following lemma says that $f_z^{\vec{\eps}, \vec{\theta}}$ does
not change much under a small complex perturbation of $(\vec{\eps},
\vec{\theta})$.

\begin{lem} \label{Cauchy-Uniformity}
For any $\delta > 0$, there exist $r_1, r_2 > 0$ such that for any $(\vec{\eps}, \vec{\theta}) \in \mathbb{C}_{\vec{\eps}_0}^{m_1}(r_1) \times \mathbb{C}_{\vec{\theta}_0}^{m_2}(r_2)$, any $z \in \mathcal{Z}$ and any $x \in W$, we have
$$
|f^{\vec{\eps}, \vec{\theta}}_{z}(x)- f^{\vec{\eps}_0, \vec{\theta}_0}_{z}(x)| \leq \delta.
$$
\end{lem}

\begin{proof}

Since
$$
f_z^{\vec{\eps}, \vec{\theta}}(x) = \frac{x \Pi^{\vec{\eps},
\vec{\theta}}(z)}{x \Pi^{\vec{\eps}, \vec{\theta}}(z)
\mathbf{1}}=\frac{x (\Pi^{\vec{\eps}, \vec{\theta}}(z)/q^{\vec{\theta}}(z|I))}{x
(\Pi^{\vec{\eps}, \vec{\theta}}(z)/q^{\vec{\theta}}(z|I)) \mathbf{1}},
$$
the lemma follows from (\ref{all-columns-comparable}) and Condition (d(i)).
\end{proof}

Recall that $\tilde{W}_D(\delta) = \{\tilde{w} \in \tilde{W} :
|\tilde{w}-w| < \delta \mbox{ for some } w \in D\}$.

\begin{lem}  \label{compact-subset-contraction}
Given any $(\vec{\eps}_0, \vec{\theta}_0) \in \Omega_1 \times
\Omega_2$, for any compact subset $D$ of $W^{\circ}$, there exists
$r_1, r_2, \delta > 0$, $0 < \rho_1 < 1$ and a positive integer
$n_0$ such that, for all $z_i^j$ with $j \geq i+n_0$ and all
$(\vec{\eps}, \vec{\theta}) \in \mathbb{C}_{\vec{\eps}_0}^{m_1}(r_1)
\times \mathbb{C}_{\vec{\theta}_0}^{m_2}(r_2)$,
$f_{z_i^j}^{\vec{\eps}, \vec{\theta}}$ is a $\rho_1$-contraction
mapping on $\tilde{W}_D(\delta)$ under the Euclidean metric.
\end{lem}

\begin{proof}
For any $z \in \mathcal{Z}$ and sufficiently small $r_1, r_2 > 0$, one checks that for any $u, v \in D$, we have
\begin{equation}  \label{key-observation}
d_H(u \Pi^{\vec{\eps}_0,\vec{\theta}_0}(z), v \Pi^{\vec{\eps}_0,\vec{\theta}_0}(z))=d_H(u \Pi^{\vec{\eps}_0}, v \Pi^{\vec{\eps}_0}).
\end{equation}
It then follows that for any $z_i^j$,
$$
d_H(f_{z_i^j}^{\vec{\eps}_0,
\vec{\theta}_0}(u), f_{z_i^j}^{\vec{\eps}_0, \vec{\theta}_0}(v)) \leq \tau(\Pi^{\vec{\eps}_0})^{j-i+1} d_H(u, v),
$$
where $\tau(\Pi^{\vec{\eps}_0})$, the Birkhoff coefficient of
$\Pi^{\vec{\eps}_0}$ as defined in (\ref{BirkhoffCoefficient}), is
strictly less than $1$. It then follows from the fact that the
Euclidean metric and the Hilbert metric are equivalent on $D$ (see
Proposition $2.1$ of~\cite{gm05}) that there exists $C > 0$ such
that for any $z_i^j$ and any $u, v \in D$,
$$
|f_{z_i^j}^{\vec{\eps}_0, \vec{\theta}_0}(u) -
f_{z_i^j}^{\vec{\eps}_0, \vec{\theta}_0}(v)|\leq C
\tau(\Pi^{\vec{\eps}_0})^{j-i+1} |u - v|.
$$
For $n_0$ sufficiently large, $\rho_0 :=
C\tau(\Pi^{\vec{\eps}_0})^{n_0} < 1$. Thus, for $j \geq i+n_0$,
$$
|D_w f_{z_i^j}^{\vec{\eps}_0, \vec{\theta}_0}| \le \rho_0 < 1
$$
for all $w \in D$.  From this and Condition (d(i))
it follows that there exists $r_1,
r_2, \delta > 0$, $0 < \rho_1 < 1$ such that
$$
|D_w f_{z_i^j}^{\vec{\eps}, \vec{\theta}}| \le \rho_1 < 1
$$
for all $w \in \tilde{W}_D(\delta)$ and
$(\vec{\eps}, \vec{\theta}) \in \mathbb{C}_{\vec{\eps}_0}^{m_1}(r_1)
\times \mathbb{C}_{\vec{\theta}_0}^{m_2}(r_2)$.  The lemma then
follows.
\end{proof}

The following lemma essentially follows from the framework in the proof of Theorem $1.1$ in~\cite{gm05}. We briefly outline the proof for completeness.

We first introduce some notation. Let
$$
\mathring{p}^{\vec{\eps},  \vec{\theta}}(z_0|z_{-n}^{-1}) \triangleq
p^{\vec{\eps}, \vec{\theta}}(z_0|z_{-n}^{-1})/
q^{\vec{\theta}}(z_0|I),
$$
where $I$ is the same as in Condition (d).

\begin{lem}  \label{Cauchy-a-b}
\begin{enumerate}
\item There is a compact subset $D$ of $W^\circ$ such that
for any $\delta > 0$, there exist $r_1, r_2 > 0$ such that for any
$(\vec{\eps}, \vec{\theta}) \in \mathbb{C}_{\vec{\eps}_0}^{m_1}(r_1)
\times \mathbb{C}_{\vec{\theta}_0}^{m_2}(r_2)$ and for all $z_{-n}^0
\in \mathcal{Z}^{n+1}$ and $-n-1 \le i \le -1$,
\begin{equation} \label{Cauchy-confined-orbit}
x_i^{\vec{\eps}, \vec{\theta}}(z_{-n}^i) \in \tilde{W}_D(\delta).
\end{equation}

\item There exist $r_1, r_2 > 0$ such that for all $z_{-n}^0 \in \mathcal{Z}^{n+1}$, $p^{\vec{\eps}, \vec{\theta}}(z_0|z_{-n}^{-1})$ is analytic on $\mathbb{C}_{\vec{\eps}_0}^{m_1}(r_1) \times \mathbb{C}_{\vec{\theta}_0}^{m_2}(r_2)$.

\item There exist $r_1, r_2 > 0$, $0 < \rho_1 < 1$ and $L_1 > 0$
such that for any two $\mathcal{Z}$-valued sequences
$\{a_{-n_1}^0\}$ and $\{b_{-n_2}^0\}$ with $a_{-n}^0=b_{-n}^0$ and
for all $(\vec{\eps}, \vec{\theta}) \in
\mathbb{C}_{\vec{\eps}_0}^{m_1}(r_1) \times
\mathbb{C}_{\vec{\theta}_0}^{m_2}(r_2)$, we have
\begin{equation} \label{Cauchy-F-rho'}
|\mathring{p}^{\vec{\eps}, \vec{\theta}}(a_0|a_{-n_1}^{-1})-\mathring{p}^{\vec{\eps}, \vec{\theta}}(b_0|b_{-n_2}^{-1})| \leq L_1 \rho_1^n.
\end{equation}
\end{enumerate}
\end{lem}

\begin{proof}

1. For a fixed $n_0 > 0$, we will consecutively reblock
$z_{-n}^{-1}$ to $\hat{z}_{-\hat{n}}^{-1}$ such that each
$\hat{z}_i$ is of the form $z_{j(i)}^{k(i)}$, where $k(i)-j(i)+1=n_0$ ($n_0$ is
determined below).

By (\ref{iter0}) and Condition (\ref{all-columns-comparable}), for any $z_{-n}^0$ and $i$,
$x_i^{\vec{\eps}_0, \vec{\theta}_0}$ (and thus
$\hat{x}_i^{\vec{\eps}_0, \vec{\theta}_0}$) belongs to a compact
subset $D$ of $W^{\circ}$. By
Lemma~\ref{compact-subset-contraction}, we can choose $r_1, r_2,
\delta >0$ sufficiently small, $n_0$ sufficiently large and  $0 <
\rho_1 < 1$ such that for all  $(\vec{\eps}, \vec{\theta}) \in
\mathbb{C}_{\vec{\eps}_0}^{m_1}(r_1) \times
\mathbb{C}_{\vec{\theta}_0}^{m_2}(r_2)$,  $f_{\hat{z}}^{\vec{\eps},
\vec{\theta}}$ is a $\rho_1$-contraction on $\tilde{W}_D(\delta)$
under the Euclidean metric.

To prove (\ref{Cauchy-confined-orbit}), it is enough to prove the
version of (\ref{Cauchy-confined-orbit}) with $x_i^{\vec{\eps},
\vec{\theta}}(z_{-n}^i)$ replaced by $\hat{x}_i^{\vec{\eps},
\vec{\theta}}(\hat{z}_{-n}^i)$ (with perhaps smaller $r_1, r_2$).

%
To see this, note that by  Lemma~\ref{Cauchy-Uniformity}, for
sufficiently small
$r_1, r_2 > 0$,  for all $\hat{z}$, $x \in W$, and  $(\vec{\eps},
\vec{\theta}) \in \mathbb{C}_{\vec{\eps}_0}^{m_1}(r_1) \times
\mathbb{C}_{\vec{\theta}_0}^{m_2}(r_2)$,
\begin{equation} \label{Cauchy-bounded-1}
|f^{\vec{\eps}, \vec{\theta}}_{\hat{z}}(x)-f^{\vec{\eps}_0,
\vec{\theta}_0}_{\hat{z}}(x)| \leq \delta (1-\rho_1),
\end{equation}
and for all  $\vec{\eps} \in \mathbb{C}_{\vec{\eps}_0}^{m_1}(r_1)$
\begin{equation} \label{Cauchy-bounded-2}
|\pi^{\vec{\eps}}-\pi(\vec{\eps}_0)| \leq \delta (1-\rho_1).
\end{equation}
Thus,
\begin{align}
\nonumber |\hat{x}^{\vec{\eps}, \vec{\theta}}_{i}-\hat{x}^{\vec{\eps}_0, \vec{\theta}_0}_{i}|
& =|f^{\vec{\eps}, \vec{\theta}}_{\hat{z}_{i}}(\hat{x}^{\vec{\eps},
\vec{\theta}}_{i-1})-f^{\vec{\eps}_0, \vec{\theta}_0}_{\hat{z}_{i}}(\hat{x}^{\vec{\eps}_0,
\vec{\theta}_0}_{i-1})|  \\
\label{Cauchy-star} & \leq |f^{\vec{\eps},
\vec{\theta}}_{\hat{z}_{i}}(\hat{x}^{\vec{\eps},
\vec{\theta}}_{i-1})-f^{\vec{\eps},
\vec{\theta}}_{\hat{z}_{i}}(\hat{x}^{\vec{\eps}_0,
\vec{\theta}_0}_{i-1})|+|f^{\vec{\eps},
\vec{\theta}}_{\hat{z}_{i}}(\hat{x}^{\vec{\eps}_0,
\vec{\theta}_0}_{i-1})-f^{\vec{\eps}_0,
\vec{\theta}_0}_{\hat{z}_{i}}(\hat{x}^{\vec{\eps}_0,
\vec{\theta}_0}_{i-1})|.
\end{align}
Then by
(\ref{Cauchy-bounded-1}) and
(\ref{Cauchy-bounded-2}), and (\ref{Cauchy-star}), we have
$$
|\hat{x}^{\vec{\eps},  \vec{\theta}}_{i}-\hat{x}^{\vec{\eps}_0,
\vec{\theta}_0}_{i}| \leq \rho_1 |\hat{x}^{\vec{\eps},
\vec{\theta}}_{i-1}-\hat{x}^{\vec{\eps}_0, \vec{\theta}_0}_{i-1}|+
\delta (1-\rho_1).
$$
So, for all $i$,
$$
|\hat{x}^{\vec{\eps}, \vec{\theta}}_{i}-\hat{x}^{\vec{\eps}_0,
\vec{\theta}_0}_{i}| \leq \delta,
$$
and thus for all $i$, we have $\hat{x}^{\vec{\eps},
\vec{\theta}}_{i} \in \tilde{W}_D(\delta)$, as desired.

2. It follows from (\ref{all-columns-comparable}) and Condition (d(i)) that for sufficiently small $r_1, r_2, \delta > 0$ and any $z \in \mathcal{Z}$, $f_z^{\vec{\eps}, \vec{\theta}}(x)$ is analytic with respect to $(\vec{\eps}, \vec{\theta}, x) \in \mathbb{C}_{\vec{\eps}_0}^{m_1}(r_1) \times
\mathbb{C}_{\vec{\theta}_0}^{m_2}(r_2) \times \tilde{W}_D(\delta)$. It then follows from this fact and the iterative nature of $x^{\vec{\eps}, \vec{\theta}}_i$ (see (\ref{iter0})) and Part $1$ that
for sufficiently small $r_1, r_2 > 0$, each $x^{\vec{\eps}, \vec{\theta}}_i$ is analytic
on $\mathbb{C}_{\vec{\eps}_0}^{m_1}(r_1) \times \mathbb{C}_{\vec{\theta}_0}^{m_2}(r_2)$. Part $2$ then immediately follows from (\ref{amalg}).

3. Applying the same reblocking as in Part 1, we write
$$
\hat{x}^{\vec{\eps}, \vec{\theta}}_{i, \hat{a}}=\hat{x}^{\vec{\eps}, \vec{\theta}}_i(\hat{a}_{-\hat{n}_1}^i)=p^{\vec{\eps}, \vec{\theta}}(y_{k(i)}=\cdot \;|\hat{a}_{-\hat{n}_1}^i),
$$
$$
\hat{x}^{\vec{\eps}, \vec{\theta}}_{i, \hat{b}}=\hat{x}^{\vec{\eps}, \vec{\theta}}_i(\hat{b}_{-\hat{n}_2}^i)=p^{\vec{\eps}, \vec{\theta}}(y_{k(i)}=\cdot \;|\hat{b}_{-\hat{n}_2}^i).
$$
Evidently we have
$$
\hat{x}^{\vec{\eps}, \vec{\theta}}_{i+1, \hat{a}}=f^{\vec{\eps}, \vec{\theta}}_{\hat{a}_{i+1}}(\hat{x}^{\vec{\eps}, \vec{\theta}}_{i, \hat{a}}),
\qquad
\hat{x}^{\vec{\eps}, \vec{\theta}}_{i+1, \hat{b}}=f^{\vec{\eps}, \vec{\theta}}_{\hat{b}_{i+1}}(\hat{x}^{\vec{\eps}, \vec{\theta}}_{i, \hat{b}}).
$$
Note that there exists a positive constant $L'_1$ such that for all $(\vec{\eps}, \vec{\theta}) \in \mathbb{C}_{\vec{\eps}_0}^{m_1}(r_1) \times \mathbb{C}_{\vec{\theta}_0}^{m_2}(r_2)$,
$$
|\hat{x}^{\vec{\eps}, \vec{\theta}}_{-\hat{n}, \hat{a}}-\hat{x}^{\vec{\eps}, \vec{\theta}}_{-\hat{n}, \hat{b}}| \leq L'_1,
$$
for all $(\vec{\eps}, \vec{\theta}) \in
\mathbb{C}_{\vec{\eps}_0}^{m_1}(r_1) \times
\mathbb{C}_{\vec{\theta}_0}^{m_2}(r_2)$, where $r_1, r_2 > 0$ are
chosen sufficiently small.  Since $f_{\hat{z}}^{\vec{\eps},
\vec{\theta}}$ is a $\rho_1$-contraction on $\tilde{W}_D(\delta)$,
we have, by Part $1$,
$$
|\hat{x}^{\vec{\eps}, \vec{\theta}}_{-1,
\hat{a}}-\hat{x}^{\vec{\eps}, \vec{\theta}}_{-1, \hat{b}}| \leq L'_1
\rho_1^{\hat{n}-1}.
$$
This, together with (\ref{all-columns-comparable}), implies that
that there exists $L_1
>0 $, independent of $n_1, n_2$, such that for all $(\vec{\eps},
\vec{\theta}) \in \mathbb{C}_{\vec{\eps}_0}^{m_1}(r_1) \times
\mathbb{C}_{\vec{\theta}_0}^{m_2}(r_2)$,
\begin{equation}
|\mathring{p}^{\vec{\eps},
\vec{\theta}}(a_0|a_{-n_1}^{-1})-\mathring{p}^{\vec{\eps},
\vec{\theta}}(b_0|b_{-n_2}^{-1})| \leq L_1 \rho_1^n.
\end{equation}

\end{proof}

The following lemma should be well-known. We sketch the proof for
completeness.
\begin{lem} \label{exercise}
Let $n_1$ and $n_2$ be positive integers and $D$ a compact domain in
$\mathbb{C}^{n_1}$. Let $f(\theta, z)$ be a jointly continuous
function on $D \times \mathbb{R}^{n_2}$.
 Assume that
\begin{equation} \label{finite}
\int_{\mathbb{R}^{n_2}}~ \sup_{\theta \in D} |f(\theta, z)| dz <
\infty.
\end{equation}
Then
\begin{enumerate}
\item  $\int_{\mathbb{R}^{n_2}} f(\theta, z) dz$ is continuous on $D$.
\item If, for each $z \in \mathcal{Z}$, $f$ is analytic on $D$, then $\int_{\mathbb{R}^{n_2}} f(\theta, z) dz$ is analytic on $D$.
\end{enumerate}
\end{lem}

\begin{proof}
Let $\Sigma$ be a compact domain in $\mathbb{R}^{n_2}$. Let
$\delta_i, i=1, 2, \ldots$ be a sequence of positive numbers
converging to $0$. Consider a sequence of partitions of $\Sigma$:
$$
\Sigma=\cup_{i=1}^{m_n} \Delta_{n, i},
$$
where $diam(\Delta_{n, i}) \leq \delta_n$ for all $1 \leq i \leq m_n$. Evidently, the corresponding Riemann sum
$$
R_n=\sum_{i=1}^{m_n} f(\theta, z_i) vol(\Delta_{n,i})
$$
(here $z_i \in \Delta_{n,i}$) is continuous in $\theta$. Then
$$
\int_{\Sigma} f(\theta, z) dz-R_n=\sum_{i=1}^{m_n}
\int_{\Delta_{n,i}} (f(\theta, z)-f(\theta, z_i)) dz.
$$
By the compactness of $D$ and $\Sigma$, we deduce that for any
$\eps_0 > 0$,  there exists $N_0$ such that for all $n \geq N_0$,
all $\theta \in D$ and all $z \in \Delta_{n, i}, i=1, 2, \ldots$,
$$
|f(\theta, z)-f(\theta, z_i)| \leq \eps_0
$$
which implies that for any $\eps > 0$, there exists $N_1$ such that for all $n \geq N_1$ and all $\theta \in D$,
$$
\left| \int_{\Sigma} f(\theta, z) dz-R_n \right| \leq \eps.
$$
In other words, $R_n$ uniformly (in $\theta \in D$) converges to
$$
\int_{\Sigma} f(\theta, z) dz,
$$
and so $\int_{\Sigma} f(\theta, z) dz$ is continuous in $\theta \in
D$.

Now, take any increasing sequence of compact sets $\Sigma_i$ whose
union is $\mathbb{R}^{n_2}$.  By (\ref{finite}), $\int_{\Sigma_i}
f(\theta, z) dz$  converges uniformly, in $\theta \in D$, to
$\int_{\mathbb{R}^{n_2}} f(\theta, z) dz$, which is therefore
continuous on $D$. This gives Part $1$.

Part $2$ follows in the same way with analyticity replacing
continuity.
\end{proof}

We are now ready for the proof of Theorem~\ref{Cauchy}.

\begin{proof}[Proof of Theorem~\ref{Cauchy}]

We first show that there exist $r_1, r_2 > 0$ such that for any $n$,
$H^{\vec{\eps}, \vec{\theta}}_n(Z)$ is analytic on
$\mathbb{C}_{\vec{\eps}_0}^{m_1}(r_1) \times
\mathbb{C}_{\vec{\theta}_0}^{m_2}(r_2)$.

For a fixed $n$, recall that
$$
H_n^{\vec{\eps},  \vec{\theta}}(Z)=-\int_{\mathcal{Z}^{n+1}}
p^{\vec{\eps}, \vec{\theta}}(z_{-n}^0) \log p^{\vec{\eps},
\vec{\theta}}(z_0|z_{-n}^{-1}) dz_{-n}^0,
$$
where
$$
p^{\vec{\eps}, \vec{\theta}}(z_{-n}^0)=\sum_{y_{-n}^0} p^{\vec{\eps}}(y_{-n}^0) \prod_{i=-n}^0 q^{\vec{\theta}}(z_i|y_i)
$$
and
$$
p^{\vec{\eps}, \vec{\theta}}(z_0|z_{-n}^{-1}) = x_{-1}^{\vec{\eps}, \vec{\theta}}(z_{-n}^{-1}) \Pi^{\vec{\eps}, \vec{\theta}}(z_0) \mathbf{1}.
$$
Now, for any $(\vec{\eps}, \vec{\theta}) \in \mathbb{C}_{\vec{\eps}_0}^{m_1}(r_1) \times \mathbb{C}_{\vec{\theta}_0}^{m_2}(r_2)$, we have
\begin{equation}
\label{pq} p^{\vec{\eps}, \vec{\theta}}(z_{-n}^0) =  \sum_{y_{-n}^0}
p^{\vec{\eps}}(y_{-n}^0) \prod_{i=-n}^0 q^{\vec{\theta}}(z_i|y_i)
\end{equation}
Since $\sum_{y_{-n}^0} p^{\vec{\eps}}(y_{-n}^0)$ is continuous and
therefore bounded as a function of $\vec{\eps} \in
\mathbb{C}_{\vec{\eps}_0}^{m_1}(r_1)$, there is a constant $K > 0$
such that
\begin{equation} \label{bounded-by-sup}
 \sup_{(\vec{\eps}, \vec{\theta})
\in \mathbb{C}_{\vec{\eps}_0}^{m_1}(r_1) \times
\mathbb{C}_{\vec{\theta}_0}^{m_2}(r_2)} |p^{\vec{\eps},
\vec{\theta}}(z_{-n}^0)| \leq
K\prod_{i=-n}^0 \sup_{(y, \vec{\theta}) \in \mathcal{Y} \times
\mathbb{C}_{\vec{\theta}_0}^{m_2}(r_2)}| q^{\vec{\theta}}(z_i|y)|.
\end{equation}
It follows from Part $1$ of Lemma~\ref{Cauchy-a-b} that for sufficiently small $r_1, r_2 > 0$, there exist
$C_1, C_2 > 0$ such that for any $z_{-n}^{0}$,
\begin{equation} \label{Cauchy-it-is-bounded}
C_1 \leq |\mathring{p}^{\vec{\eps}, \vec{\theta}}(z_0|z_{-n}^{-1})| \leq C_2,
\end{equation}
which implies that for some $C_3 > 0$,
\begin{equation} \label{Cauchy-log-is-bounded}
|\log \mathring{p}^{\vec{\eps}, \vec{\theta}}(z_0|z_{-n}^{-1})| \leq C_3.
\end{equation}
It then follows  that  for any $(\vec{\eps}, \vec{\theta}) \in
\mathbb{C}_{\vec{\eps}_0}^{m_1}(r_1) \times
\mathbb{C}_{\vec{\theta}_0}^{m_2}(r_2)$,
\begin{align}
\nonumber \int_{\mathcal{Z}^{n+1}} & \sup_{(\vec{\eps}, \vec{\theta})
 \in \mathbb{C}_{\vec{\eps}_0}^{m_1}(r_1) \times \mathbb{C}_{\vec{\theta}_0}^{m_2}(r_2)} \left|
  p^{\vec{\eps}, \vec{\theta}}(z_{-n}^0) \log p^{\vec{\eps}, \vec{\theta}}(z_0|z_{-n}^{-1}) \right| dz_{-n}^0  \\
\nonumber &= \int_{\mathcal{Z}^{n+1}} \sup_{(\vec{\eps}, \vec{\theta}) \in \mathbb{C}_{\vec{\eps}_0}^{m_1}(r_1)
\times \mathbb{C}_{\vec{\theta}_0}^{m_2}(r_2)} \left| p^{\vec{\eps}, \vec{\theta}}(z_{-n}^0)
\log q^{\vec{\theta}}(z_0|I)+ p^{\vec{\eps}, \vec{\theta}}(z_{-n}^{0}) \log \mathring{p}^{\vec{\eps},
 \vec{\theta}}(z_0|z_{-n}^{-1}) \right| dz_{-n}^0 \\
\nonumber &\leq \int_{\mathcal{Z}^{n+1}} \sup_{ \vec{\theta} \in
\mathbb{C}_{\vec{\theta}_0}^{m_2}(r_2)} K \left|
\prod_{i=-n}^0 q^{\vec{\theta}}(z_i|y_i)\right| \left|
\log q^{\vec{\theta}}(z_0|I) \right| dz_{-n}^0 \\
 \label{Cauchy-absolute-value-bounded} & +\int_{\mathcal{Z}^{n+1}} \sup_{(\vec{\eps}, \vec{\theta}) \in \mathbb{C}_{\vec{\eps}_0}^{m_1}(r_1) \times \mathbb{C}_{\vec{\theta}_0}^{m_2}(r_2)} |p^{\vec{\eps}, \vec{\theta}}(z_{-n}^0)| |\log \mathring{p}^{\vec{\eps}, \vec{\theta}}(z_0|z_{-n}^{-1})| dz_{-n}^0 < \infty.
\end{align}
(for the first term we have used (\ref{max-of-three}(i)) and
(\ref{first-term-analytic-1}); for the second term, we have used
(\ref{max-of-three}(i)), (\ref{bounded-by-sup}) and
(\ref{Cauchy-log-is-bounded})).
By lemma~\ref{exercise} (Part $2$),
$$
H_n^{\vec{\eps}, \vec{\theta}}(Z_0|Z_{-n}^{-1}) =
\int_{\mathcal{Z}^{n+1}} p^{\vec{\eps}, \vec{\theta}}(z_{-n}^0) \log
p^{\vec{\eps}, \vec{\theta}}(z_0|z_{-n}^{-1}) dz_{-n}^0,
$$
is analytic on $\mathbb{C}_{\vec{\eps}_0}^{m_1}(r_1) \times
\mathbb{C}_{\vec{\theta}_0}^{m_2}(r_2)$.

Now, to prove the theorem, we only need to prove that there exist $r_1, r_2 > 0$ such that $H^{\vec{\eps}, \vec{\theta}}_n(Z)$ uniformly converges on $\mathbb{C}_{\vec{\eps}_0}^{m_1}(r_1) \times \mathbb{C}_{\vec{\theta}_0}^{m_2}(r_2)$ as $n \to \infty$. First, we observe that
\begin{align*}
\hspace{-1cm} |H^{\vec{\eps}, \vec{\theta}}_{n+1}(Z)-H^{\vec{\eps}, \vec{\theta}}_n(Z)| &=\left| \int_{\mathcal{Z}^{n+2}} p^{\vec{\eps}, \vec{\theta}}(z_{-n-1}^{0}) \log p^{\vec{\eps}, \vec{\theta}}(z_0|z_{-n-1}^{-1}) dz_{-n-1}^0 - \int_{\mathcal{Z}^{n+1}} p^{\vec{\eps}, \vec{\theta}}(z_{-n}^{0}) \log p^{\vec{\eps}, \vec{\theta}}(z_0|z_{-n}^{-1}) dz_{-n}^0 \right| \\
&=\left| \int_{\mathcal{Z}^{n+2}} p^{\vec{\eps}, \vec{\theta}}(z_{-n-1}^{0}) \log \mathring{p}^{\vec{\eps}, \vec{\theta}}(z_0|z_{-n-1}^{-1}) dz_{-n-1}^0 - \int_{\mathcal{Z}^{n+1}} p^{\vec{\eps}, \vec{\theta}}(z_{-n}^{0}) \log \mathring{p}^{\vec{\eps}, \vec{\theta}}(z_0|z_{-n}^{-1}) dz_{-n}^0 \right| \\
&=\left| \int_{\mathcal{Z}^{n+2}} p^{\vec{\eps}, \vec{\theta}}(z_{-n-1}^{0}) (\log \mathring{p}^{\vec{\eps}, \vec{\theta}}(z_0|z_{-n-1}^{-1})-\log \mathring{p}^{\vec{\eps}, \vec{\theta}}(z_0|z_{-n}^{-1})) dz_{-n-1}^0 \right|.
\end{align*}
Fix $(\vec{\eps}, \vec{\theta}) \in \mathbb{C}_{\vec{\eps}_0}^{m_1}(r_1) \times \mathbb{C}_{\vec{\theta}_0}^{m_2}(r_2)$. Then, by (\ref{Cauchy-F-rho'}), (\ref{Cauchy-it-is-bounded}), (\ref{Cauchy-log-is-bounded}), we have, for some $0 < \rho_1 < 1$, $L_1', L_1 > 0$,
\begin{align*}
| p^{\vec{\eps}, \vec{\theta}}(z_{-n-1}^{0}) (\log \mathring{p}^{\vec{\eps}, \vec{\theta}}(z_0|z_{-n-1}^{-1})-\log \mathring{p}^{\vec{\eps}, \vec{\theta}}(z_0|z_{-n}^{-1})) | & \leq L'_1 \left| p^{\vec{\eps}, \vec{\theta}}(z_{-n-1}^{0}) (\mathring{p}^{\vec{\eps}, \vec{\theta}}(z_0|z_{-n-1}^{-1})-\mathring{p}^{\vec{\eps}, \vec{\theta}}(z_0|z_{-n}^{-1}))\right| \\
& \leq L'_1 |p^{\vec{\eps}, \vec{\theta}}(z_{-n-1}^{0})| L_1 \rho_1^n.
\end{align*}
Notice that for any given $\delta > 0$, there exist $r_1, r_2 > 0$ such that for all $\vec{\eps} \in \mathbb{C}_{\vec{\eps}_0}^{m_1}(r_1)$,
$$
|\pi_{y_{-n}}^{\vec{\eps}}| \leq (1+\delta)
\pi_{y_{-n}}^{\vec{\eps}_0}, \qquad |\pi_{y_i y_{i+1}}^{\vec{\eps}}|
\leq (1+\delta) \pi_{y_i y_{i+1}}^{\vec{\eps}_0},
$$
and for any $y \in \mathcal{Y}$ and all $\vec{\theta} \in \mathbb{C}_{\vec{\theta}_0}^{m_2}(r_2)$,
$$
\int_{\mathcal{Z}} |q^{\vec{\theta}}(z|y)| dz \leq (1+\delta) \int_{\mathcal{Z}} q^{\vec{\theta}_0}(z|y) dz=1+\delta,
$$
(here we have used the fact that $\int_{\mathcal{Z}}
|q^{\vec{\theta}}(z|y)| dz$ is a continuous function of
$\vec{\theta} \in \mathbb{C}_{\vec{\theta}_0}^{m_2}(r_2)$; this
follows from Lemma~\ref{exercise} (Part $1$)).
It then follows from (\ref{bounded-by-sup}) that
$$
\int_{\mathcal{Z}^{n+1} }|p^{\vec{\eps}, \vec{\theta}}(z_{-n-1}^{-1})| dz_{-n-1}^{-1} \leq (1+\delta)^{2(n+2)}.
$$

By choosing $\delta > 0$ sufficiently small, we can combine all the
relevant inequalities above to obtain some $L
> 0$ and some $0 < \rho < 1$ such that for all $(\vec{\eps},
\vec{\theta}) \in \mathbb{C}_{\vec{\eps}_0}^{m_1}(r_1) \times
\mathbb{C}_{\vec{\theta}_0}^{m_2}(r_2)$,
$$
|H^{\vec{\eps}, \vec{\theta}}_{n+1}(Z)-H^{\vec{\eps}, \vec{\theta}}_n(Z)|
\leq \int_{\mathcal{Z}^{n+2}} | p^{\vec{\eps}, \vec{\theta}}(z_{-n-1}^{0}) (\log \mathring{p}^{\vec{\eps}, \vec{\theta}}(z_0|z_{-n-1}^{-1})-\log \mathring{p}^{\vec{\eps}, \vec{\theta}}(z_0|z_{-n}^{-1})) | dz_{-n-1}^0
\leq L \rho^n,
$$
which implies the uniform convergence of $H^{\vec{\eps},
\vec{\theta}}_n(Z)$ on $\mathbb{C}_{\vec{\eps}_0}^{m_1}(r_1) \times
\mathbb{C}_{\vec{\theta}_0}^{m_2}(r_2)$ as $n$ tends to infinity,
and thus the analyticity of $H^{\vec{\eps}, \vec{\theta}}(Z)$ around
$(\vec{\eps}_0, \vec{\theta}_0)$.


\end{proof}

\section{Proof of Theorem~\ref{Gaussian}} \label{IV}

The following lemma is an analog of Lemma~\ref{Cauchy-Uniformity}.

\begin{lem} \label{Gaussian-Uniformity}
For any $\delta > 0$, there exist $r_1, r_2 > 0$ such that for any $(\vec{\eps}, \vec{\theta}) \in \mathbb{C}_{\vec{\eps}_0}^{m_1}(r_1) \times \mathbb{C}_{\vec{\theta}_0}^{m_2}(r_2)$, any $z \in \mathcal{Z}$ and any $x \in W$, we have
$$
|f^{\vec{\eps}, \vec{\theta}}_{z}(x)- f^{\vec{\eps}_0, \vec{\theta}_0}_{z}(x)| \leq \delta.
$$
\end{lem}

\begin{proof}
Note that for any $x \in W$,
\begin{equation} \label{Gaussian-divided-by-I}
f_z^{\vec{\eps}, \vec{\theta}}(x) = \frac{x \Pi^{\vec{\eps},
\vec{\theta}}(z)}{x \Pi^{\vec{\eps}, \vec{\theta}}(z)
\mathbf{1}}=\frac{x (\Pi^{\vec{\eps}, \vec{\theta}}(z)/q^{\vec{\theta}}(z|I))}{x
(\Pi^{\vec{\eps}, \vec{\theta}}(z)/q^{\vec{\theta}}(z|I)) \mathbf{1}}.
\end{equation}
It then follows from (\ref{one-column-dominant})
that for sufficiently small $r_1, r_2 > 0$, there exists a compact
subset $\Sigma \subset \mathcal{Z}$ such that for any $(\vec{\eps},
\vec{\theta}) \in \mathbb{C}_{\vec{\eps}_0}^{m_1}(r_1) \times
\mathbb{C}_{\vec{\theta}_0}^{m_2}(r_2)$, any $z \not \in \Sigma$ and
any $x \in W$, $x (\Pi^{\vec{\eps}, \vec{\theta}}(z)/q^{
\vec{\theta}}(z|I)) \mathbf{1}$ is bounded away from $0$. On the
other hand, by the compactness of $\Sigma$, we deduce that there
exist $r_1, r_2 > 0$ such that for any $(\vec{\eps}, \vec{\theta})
\in \mathbb{C}_{\vec{\eps}_0}^{m_1}(r_1) \times
\mathbb{C}_{\vec{\theta}_0}^{m_2}(r_2)$, any $z \in \mathcal{Z}$ and
any $x \in W$, $x (\Pi^{\vec{\eps},
\vec{\theta}}(z)/q^{\vec{\theta}}(z|I)) \mathbf{1}$ is bounded away
from $0$. The lemma then follows from Condition (d(i)).
\end{proof}

\begin{lem} \label{M-two-part-contraction}
\begin{enumerate}
\item Given any $(\vec{\eps}_0, \vec{\theta}_0) \in \Omega_1 \times \Omega_2$, there exist $r_1, r_2, \delta > 0$, $0 < \rho_1 < 1$ and a compact subset $\Sigma \subset \mathcal{Z}$ such that for any $z \not \in \Sigma$ and all $(\vec{\eps}, \vec{\theta}) \in \mathbb{C}_{\vec{\eps}_0}^{m_1}(r_1) \times \mathbb{C}_{\vec{\theta}_0}^{m_2}(r_2)$, $f_{z}^{\vec{\eps}, \vec{\theta}}$ is a $\rho_1$-contraction mapping on $\tilde{W}_W(\delta)$ under the Euclidean metric.

\item Given any $(\vec{\eps}_0, \vec{\theta}_0) \in \Omega_1 \times \Omega_2$, for any
compact subset $\Sigma \subset \mathcal{Z}$, there exist $r_1, r_2,
\delta > 0$, $0 < \rho_1 < 1$ and a positive integer $n_0$ such
that, for all $z_i^j$ with $z_i \in \Sigma$ and $j \geq i+n_0$ and
all $(\vec{\eps}, \vec{\theta}) \in
\mathbb{C}_{\vec{\eps}_0}^{m_1}(r_1) \times
\mathbb{C}_{\vec{\theta}_0}^{m_2}(r_2)$, $f_{z_i^j}^{\vec{\eps},
\vec{\theta}}$ is a $\rho_1$-contraction mapping on
$\tilde{W}_W(\delta)$ under the Euclidean metric.
\end{enumerate}
\end{lem}

\begin{proof}
\textbf{1.} For any $x \in W$,
$$
f_z^{\vec{\eps}_0, \vec{\theta}_0}(x) = \frac{x \Pi^{\vec{\eps}_0,
\vec{\theta}_0}(z)}{x \Pi^{\vec{\eps}_0, \vec{\theta}_0}(z)
\mathbf{1}}=\frac{x (\Pi^{\vec{\eps}_0,
\vec{\theta}_0}(z)/q^{\vec{\theta}_0}(z|I))}{x (\Pi^{\vec{\eps}_0,
\vec{\theta}_0}(z)/q^{\vec{\theta}_0}(z|I)) \mathbf{1}}.
$$
Observe, by the quotient rule, that $D_x f_z^{\vec{\eps}_0,
\vec{\theta}_0}(x)$ is a rational function of entries of $x$ with
coefficients that are products of quantities of the form
$\pi^{\vec{\eps}}_{ij}$ and $q^{\vec{\theta}}(z|j)/q^{\vec{\theta}}(z|I)$.  By
(\ref{one-column-dominant}) (with $\vec{\theta}$ set to be
$\vec{\theta}_0$), we obtain the result at $\vec{\eps} =
\vec{\eps}_0$ and $\vec{\theta} = \vec{\theta}_0$. For general
$(\vec{\eps}, \vec{\theta}) \in \mathbb{C}_{\vec{\eps}_0}^{m_1}(r_1)
\times \mathbb{C}_{\vec{\theta}_0}^{m_2}(r_2)$, the result holds by
Condition (d(i)).

\textbf{2.} For any two points $x, y \in W$, we have, at $(\vec{\eps}_0, \vec{\theta}_0)$,
$$
\hspace{-1cm} \frac{d_E(f_{z_i^j}(x), f_{z_i^j}(y))}{d_E(x, y)} = \frac{d_E(f_{z_i^j}(x), f_{z_i^j}(y))}{d_H(f_{z_i^j}(x), f_{z_i^j}(y))} \left(\prod_{k=i}^{j-1} \frac{d_H(f_{z_i^{k+1}}(x), f_{z_i^{k+1}}(y))}{d_H(f_{z_i^{k}}(x), f_{z_i^{k}}(y))}\right) \frac{d_H(f_{z_i}(x), f_{z_i}(y))}{d_E(f_{z_i}(x), f_{z_i}(y))} \frac{d_E(f_{z_i}(x), f_{z_i}(y))}{d_E(x, y)},
$$
where $d_E$ denotes the Euclidean metric. First, since the Hilbert
metric dominates the Euclidean metric up to a multiplicative factor,
there exists $C_1 > 0$ independent of all $z_i^j$ such that for all
$k=i+1, \ldots, j,$
$$
\frac{d_E(f_{z_i^j}(x), f_{z_i^j}(y))}{d_H(f_{z_i^j}(x), f_{z_i^j}(y))} \leq C_1.
$$
By the fact that $f_z$ is a contraction mapping on $W^\circ$ under
the Hilbert metric, there exists $0 < \rho_0 < 1$ independent of all
$z_i^j$ such that
$$
\frac{d_H(f_{z_i^k}(x), f_{z_i^k}(y))}{d_H(f_{z_i^{k-1}}(x), f_{z_i^{k-1}}(y))} \leq \rho_0.
$$
On the other hand, since the Hilbert and Euclidean metrics are
equivalent on compact subsets of $W^\circ$, and  $f_{z_i}(W)$ is
compact, there exists $C_2
> 0$, which only depends on $\Sigma$, such that
$$
\frac{d_H(f_{z_i}(x), f_{z_i}(y))}{d_E(f_{z_i}(x), f_{z_i}(y))} \leq C_2.
$$
Finally, it can be easily checked that there exists $C_3 > 0$, which only depends on $\Sigma$, such that
$$
\frac{d_E(f_{z_i}(x), f_{z_i}(y))}{d_E(x, y)} \leq C_3.
$$
Combining all the above inequalities, we deduce that there exists $C_4 > 0$ such that
$$
\frac{d_E(f_{z_i^j}(x), f_{z_i^j}(y))}{d_E(x, y)} \leq C_4 \rho_0^{j-i},
$$
which, together with Condition (d(i)),
implies Part $2$.
\end{proof}

\begin{lem}  \label{Gaussian-a-b}
\begin{enumerate}
\item For any $\delta > 0$, there exists $r_1, r_2 > 0$ such that for any $(\vec{\eps}, \vec{\theta}) \in \mathbb{C}_{\vec{\eps}_0}^{m_1}(r_1) \times \mathbb{C}_{\vec{\theta}_0}^{m_2}(r_2)$ and for all $z_{-n}^0 \in \mathcal{Z}^{n+1}$ and $-n-1 \le i \le -1$,
\begin{equation} \label{Gaussian-confined-orbit}
x_i^{\vec{\eps}, \vec{\theta}}(z_{-n}^i) \in \tilde{W}_W(\delta).
\end{equation}

\item There exist $r_1, r_2 > 0$ such that for all $z_{-n}^0 \in \mathcal{Z}^{n+1}$, $p^{\vec{\eps}, \vec{\theta}}(z_0|z_{-n}^{-1})$ is analytic on $\mathbb{C}_{\vec{\eps}_0}^{m_1}(r_1) \times \mathbb{C}_{\vec{\theta}_0}^{m_2}(r_2)$.

\item There exist $r_1, r_2 > 0$, $0 < \rho_1 < 1$ and $L_1 > 0$ such that for any two $\mathcal{Z}$-valued sequences $\{a_{-n_1}^0\}$ and $\{b_{-n_2}^0\}$ with $a_{-n}^0=b_{-n}^0$ and for all $(\vec{\eps}, \vec{\theta}) \in \mathbb{C}_{\vec{\eps}_0}^{m_1}(r_1) \times \mathbb{C}_{\vec{\theta}_0}^{m_2}(r_2)$, we have
\begin{equation}
|\mathring{p}^{\vec{\eps}, \vec{\theta}}(a_0|a_{-n_1}^{-1})-\mathring{p}^{\vec{\eps}, \vec{\theta}}(b_0|b_{-n_2}^{-1})| \leq L_1 \rho_1^n.
\end{equation}
\end{enumerate}
\end{lem}

\begin{proof}
1. For a given compact subset $\Sigma \subset \mathcal{Z}$, consecutively reblock $z_{-n}^{-1}$ to a $\hat{z}_{-\hat{n}}^{-1}$ such that each $\hat{z}_i=z_{j(i)}^{k(i)}$ is of
\begin{itemize}
\item Type I: $j(i)=k(i)$ and $z_{j(i)} \not \in \Sigma$; or
\item Type II: $z_{j(i)} \in \Sigma$ and $k(i)-j(i) \geq n_0$, where $n_0$ is as in Part $2$ of Lemma~\ref{M-two-part-contraction}.
\end{itemize}
By Lemma~\ref{M-two-part-contraction}, we can choose $r_1, r_2, \delta > 0$ sufficiently small such that
there exists $0 < \rho_1 < 1$ such that for all  $(\vec{\eps}, \vec{\theta}) \in \mathbb{C}_{\vec{\eps}_0}^{m_1}(r_1) \times \mathbb{C}_{\vec{\theta}_0}^{m_2}(r_2)$,  $f_{\hat{z}}^{\vec{\eps}, \vec{\theta}}$ is a $\rho_1$-contraction on $\tilde{W}_W(\delta)$ under the Euclidean metric.

To prove (\ref{Gaussian-confined-orbit}), it is enough to prove the
version of (\ref{Gaussian-confined-orbit}) with $x_i^{\vec{\eps},
\vec{\theta}}(z_{-n}^i)$ replaced by $\hat{x}_i^{\vec{\eps},
\vec{\theta}}(\hat{z}_{-n}^i)$ (with perhaps smaller $r_1, r_2$).

Now, choose $r_1, r_2 > 0$ so small (the existence of $r_1, r_2$ is guaranteed by Lemma~\ref{Gaussian-Uniformity}) such that for any $\hat{z}$, for all $x \in W$, all $(\vec{\eps}, \vec{\theta}) \in \mathbb{C}_{\vec{\eps}_0}^{m_1}(r_1) \times \mathbb{C}_{\vec{\theta}_0}^{m_2}(r_2)$,
\begin{equation}
\label{Gaussian-bounded-1}
|f^{\vec{\eps}, \vec{\theta}}_{\hat{z}}(x)-f^{\vec{\eps}_0, \vec{\theta}_0}_{\hat{z}}(x)| \leq \delta (1-\rho_1),
\end{equation}
and for all $\vec{\eps} \in \mathbb{C}_{\vec{\eps}_0}^{m_1}(r_1)$,
\begin{equation}
\label{Gaussian-bounded-2} |\pi^{\vec{\eps}}-\pi^{\vec{\eps}_0}| \leq
\delta (1-\rho_1).
\end{equation}
We then deduce that
\begin{align}
\nonumber |\hat{x}^{\vec{\eps}, \vec{\theta}}_{i+1}-\hat{x}^{\vec{\eps}_0, \vec{\theta}_0}_{i+1}| &=|f^{\vec{\eps}, \vec{\theta}}_{\hat{z}_{i+1}}(\hat{x}^{\vec{\eps}, \vec{\theta}}_i)-f^{\vec{\eps}_0, \vec{\theta}_0}_{\hat{z}_{i+1}}(\hat{x}^{\vec{\eps}_0, \vec{\theta}_0}_i)| \\
\label{Gaussian-star} & \leq |f^{\vec{\eps}, \vec{\theta}}_{\hat{z}_{i+1}}(\hat{x}^{\vec{\eps}, \vec{\theta}}_i)-f^{\vec{\eps}, \vec{\theta}}_{\hat{z}_{i+1}}(\hat{x}^{\vec{\eps}_0, \vec{\theta}_0}_i)|+|f^{\vec{\eps}, \vec{\theta}}_{\hat{z}_{i+1}}(\hat{x}^{\vec{\eps}_0, \vec{\theta}_0}_i)-f^{\vec{\eps}_0, \vec{\theta}_0}_{\hat{z}_{i+1}}(\hat{x}^{\vec{\eps}_0, \vec{\theta}_0}_i)|.
\end{align}
Then by (\ref{Gaussian-bounded-1}) and
(\ref{Gaussian-bounded-2}), and (\ref{Gaussian-star}), for $i > -\hat{n}-1$, we have
$$
|\hat{x}^{\vec{\eps}, \vec{\theta}}_{i+1}-\hat{x}^{\vec{\eps}_0, \vec{\theta}_0}_{i+1}| \leq \rho_1 |\hat{x}^{\vec{\eps}, \vec{\theta}}_i-\hat{x}^{\vec{\eps}_0, \vec{\theta}_0}_i|+ \delta (1-\rho_1).
$$
So, for all $i$,
$$
|\hat{x}^{\vec{\eps}, \vec{\theta}}_{i+1}-\hat{x}^{\vec{\eps}_0, \vec{\theta}_0}_{i+1}| \leq \delta,
$$
and thus for all $i$, we have $\hat{x}^{\vec{\eps}, \vec{\theta}}_{i+1} \in \tilde{W}_W(\delta)$, as desired.

2. It follows from the same dichotomy argument in the proof of Lemma~\ref{Gaussian-Uniformity} that for sufficiently small $r_1, r_2, \delta > 0$ and any $z \in \mathcal{Z}$, $f_z^{\vec{\eps}, \vec{\theta}}(x)$ is analytic with respect to $(\vec{\eps}, \vec{\theta}, x) \in \mathbb{C}_{\vec{\eps}_0}^{m_1}(r_1) \times \mathbb{C}_{\vec{\theta}_0}^{m_2}(r_2) \times \tilde{W}_W(\delta)$. It follows from this fact, the iterative nature of $x^{\vec{\eps}, \vec{\theta}}_i$ (see (\ref{iter0})) and Part $1$ that
for sufficiently small $r_1, r_2 > 0$, each $x^{\vec{\eps}, \vec{\theta}}_i$ is analytic
on $\mathbb{C}_{\vec{\eps}_0}^{m_1}(r_1) \times \mathbb{C}_{\vec{\theta}_0}^{m_2}(r_2)$. Part $2$ then immediately follows from (\ref{amalg}).

3. It follows from (\ref{Gaussian-confined-orbit}) and a similar argument as in the proof of Part $2$ of Lemma~\ref{Cauchy-a-b}.
\end{proof}

We are now ready for the proof of Theorem~\ref{Gaussian}.

\begin{proof}[Proof of Theorem~\ref{Gaussian}]

By (\ref{Gaussian-confined-orbit}), for sufficiently small $r_1, r_2
> 0$, there exist $C_1, C_2 > 0$ such that for any $z_{-n}^{0}$ and
any $(\vec{\eps}, \vec{\theta}) \in
\mathbb{C}_{\vec{\eps}_0}^{m_1}(r_1) \times
\mathbb{C}_{\vec{\theta}_0}^{m_2}(r_2)$,
\begin{equation} \label{Gaussian-it-is-bounded}
C_1 \leq |\mathring{p}^{\vec{\eps}, \vec{\theta}}(z_0|z_{-n}^{-1})| \leq C_2,
\end{equation}
which implies that for some $C_3 > 0$,
\begin{equation} \label{Gaussian-log-is-bounded}
|\log \mathring{p}^{\vec{\eps}, \vec{\theta}}(z_0|z_{-n}^{-1})| \leq C_3.
\end{equation}
Then, using the same argument as in the proof of Theorem~\ref{Cauchy}, we can show that there exist $r_1, r_2 > 0$ such that for any $n$, $H^{\vec{\eps}, \vec{\theta}}_n(Z)$ is analytic on $\mathbb{C}_{\vec{\eps}_0}^{m_1}(r_1) \times \mathbb{C}_{\vec{\eps}_0}^{m_2}(r_2)$. So, to prove the theorem, we only need to prove that there exist $r_1, r_2 > 0$ such that the $H^{\vec{\eps}, \vec{\theta}}_n(Z)$ uniformly converges on $\mathbb{C}_{\vec{\eps}_0}^{m_1}(r_1) \times \mathbb{C}_{\vec{\theta}_0}^{m_2}(r_2)$ as $n \to \infty$. This follows from (\ref{Gaussian-it-is-bounded}), (\ref{Gaussian-log-is-bounded}) and Lemma~\ref{Gaussian-a-b}, and a completely parallel argument as in the proof of Theorem~\ref{Cauchy}.

\end{proof}

With the following example, we show that for Gaussian channels,
$H(Z)$ need not be analytic even as a function of the channel
parameters alone, when the largest $\sigma_i$ is not unique.
\begin{exmp}
\label{example} Consider an additive Gaussian channel parameterized
as in (\ref{q-z-y-Gaussian}) with the binary input alphabet
$\mathcal{Y}=\{1, 2\}$. Assume that the input $Y$ is an i.i.d. process with
$$
P(Y_1=1)=P(Y_1=2)=1/2;
$$
and assume that
$$
q(z|1)=\frac{1}{\sqrt{2\pi} \sigma_1}e^{-(z+1)^2/\sigma_1^2}, ~~~ q(z|2)=\frac{1}{\sqrt{2\pi} \sigma_2}e^{-(z-1)^2/\sigma_2^2}.
$$
We then have
$$
p(z)=P(Y=1) q(z|1)+P(Y=2)q(z|2)=\frac{1}{2} \frac{1}{\sqrt{2\pi}
\sigma_1}e^{-(z+1)^2/\sigma_1^2}+\frac{1}{2} \frac{1}{\sqrt{2\pi}
\sigma_2}e^{-(z-1)^2/\sigma_2^2}.
$$
We claim that for any fixed $\sigma > 0$, analyticity of $H(Z)$ as a
function of $(\sigma_1, \sigma_2)$ fails at $(\sigma_1,
\sigma_2)=(\sigma, \sigma)$. To see this, we fix $\sigma_1 =
\sigma$, and we show that $H(Z)$ is not analytic
with respect to $\sigma_2$ at $\sigma_2=\sigma$. Note that for any real
$\sigma_2$,
\begin{align}
\hspace{-1cm} H(Z)&=-\int_{-\infty}^{\infty} p(z) \log p(z) dz \notag\\
    &= -\int_{0}^{\infty} p(z) \left(\log \frac{1}{2} \frac{1}{\sqrt{2\pi} \sigma_2}
    e^{-(z-1)^2/\sigma_2^2} \right) dz - \int_{-\infty}^{0} p(z) \left(\log \frac{1}{2} \frac{1}{\sqrt{2\pi} \sigma}
    e^{-(z-1)^2/\sigma^2} \right) dz \label{log-Phi-1}\\
    &\phantom{==} -\int_{0}^{\infty} p(z) \log \left(1 + \Phi_z(\sigma_2) \right) dz - \int_{-\infty}^{0} p(z) \log \left(1 + \Phi_z^{-1}(\sigma_2) \right) dz, \label{log-Phi-2}
\end{align}
where
\begin{equation}
\label{Phiz} \Phi_z(\sigma_2) \triangleq \frac{\sigma_2}{\sigma}
e^{(z-1)^2/\sigma_2^2-(z+1)^2/\sigma^2} \mbox{ and } \Phi_z^{-1}(\sigma_2) \triangleq \frac{1}{\Phi_z(\sigma_2)}.
\end{equation}
We note that (\ref{log-Phi-1}) is analytic as a function of $\sigma_2$ at $\sigma_2=\sigma$. To see this, observe that the first term of (\ref{log-Phi-1}) can be further computed as
$$
\hspace{-1cm} -\log (2 \sqrt{2\pi} \sigma_2) \int_{0}^{\infty} p(z) dz-\int_{0}^{\infty} \frac{1}{2 \sqrt{2 \pi} \sigma \sigma_2^2} (z-1)^2 e^{-(z+1)^2/\sigma^2} dz-\int_{0}^{\infty} \frac{1}{2 \sqrt{2 \pi} \sigma_2^3} (z-1)^2 e^{-(z-1)^2/\sigma_2^2} dz,
$$
which is analytic at $\sigma_2 =\sigma$ since each of the three terms above is analytic at $\sigma_2=\sigma$ (for the second or third term, regard $\sigma_2$ as a complex variable and use the exponentially-decaying tail of the integrand). With a similar argument applied to the second term of (\ref{log-Phi-1}), we can then establish the analyticity of (\ref{log-Phi-1}).  So, to prove $H(Z)$ is not analytic at $\sigma_2=\sigma$, it suffices to show that (\ref{log-Phi-2}) is not analytic at $\sigma_2=\sigma$.

Suppose, by way of contradiction, that (\ref{log-Phi-2}) is analytic at $\sigma$, or equivalently, the following function of $\omega$
$$
\int_{0}^{\infty} \left( \frac{1}{2} \frac{\sigma^{-1}}{\sqrt{2\pi}}e^{-(z+1)^2 \sigma^{-2}}+\frac{1}{2} \frac{\omega}{\sqrt{2\pi}}e^{-(z-1)^2 \omega^2} \right) \log \left( 1+ \Phi_z(1/\omega) \right)dz
$$
\begin{equation} \label{variant-of-the-second-term}
+\int_{-\infty}^{0} \left( \frac{1}{2} \frac{\sigma^{-1}}{\sqrt{2\pi}}e^{-(z+1)^2 \sigma^{-2}}+\frac{1}{2} \frac{\omega}{\sqrt{2\pi}}e^{-(z-1)^2 \omega^2} \right) \log \left( 1+ \Phi_z^{-1}(1/\omega) \right)dz
\end{equation}
is analytic at $\sigma^{-1} \in \overline{\mathbb{C}_{\sigma^{-1}}(r)}$ (the closure of the $r$-neighborhood of $\sigma^{-1}$ in $\mathbb{C}$) for some $r > 0$, where, recalling from (\ref{Phiz}),
$$
\Phi_z(1/\omega)=\frac{\sigma^{-1}}{\omega} e^{(z-1)^2 \omega^2-(z+1)^2 \sigma^{-2}}.
$$
Then, by uniqueness, the analytic extension of (\ref{variant-of-the-second-term}) to $\overline{\mathbb{C}_{\sigma^{-1}}(r)}$ would agree with any analytic extension  along the circle $\{\sigma^{-1}+r e^{i \alpha}: \alpha \in [-\pi/2, 3\pi/2]\}$ (from $\alpha=-\pi/2$ to $\alpha=3\pi/2$). Such an analytic extension is obtained by regarding  $\omega$ as a complex variable on the circle (this is a valid analytic extension by virtue of the exponentially-decaying tails of the integrands in (\ref{variant-of-the-second-term})). Here, we remark that for any $r > 0$ and $\alpha$, there are at most two ``singular'' $z$ (note that the following inequality boils down to a system of two quadratic equations in $z$) such that
$$
\Phi_z(1/(\sigma^{-1} + r e^{i \alpha})) = -1,
$$
which means $\log \left( 1+ \Phi_z(1/(\sigma^{-1} + r e^{i \alpha})) \right)$ or $\log \left( 1+ \Phi_z^{-1}(1/(\sigma^{-1} + r e^{i \alpha})) \right)$ would ``blow up'' at such $z$. However, an easy bounding argument (roughly speaking, the two ``blowing up'' terms will only do so ``slowly'') yields that during the analytic extension, (\ref{variant-of-the-second-term}) is still well-defined with the presence of such singular $z$, and so the above analytic extension is indeed valid.

Next, we will find a contradiction by showing that the analytic extension of (\ref{variant-of-the-second-term}) disagrees at $\alpha=-\pi/2$ and $\alpha=3\pi/2$. Setting $\omega = \sigma^{-1}+r e^{i \alpha}$, we then have
$$
\frac{\sigma^{-1}}{w}=\frac{\sigma^{-1}}{\sigma^{-1}+r \cos \alpha + i r \sin \alpha}=\frac{\sigma^{-2}+\sigma^{-1} r \cos \alpha-i \sigma^{-1} r \sin \alpha}{\sigma^{-2}+2 \sigma^{-1} r \cos \alpha+ r^2} \triangleq e^{a(r, \alpha)+ i b(r, \alpha)},
$$
where one can easily check that
$$
a(r, \alpha)=O(r), ~~~ b(r, \alpha)=O(r), ~~~ \frac{\partial b(r, \alpha)}{\partial \alpha}=O(r).
$$
Then, some straightforward computations yield that
$$
\Phi_z(1/\omega)=e^{A(z, r, \alpha)} e^{i B(z, r, \alpha)},
$$
where
$$
A(z, r, \alpha) \triangleq 2 (z-1)^2 \sigma^{-1} r \cos \alpha+(z-1)^2 r^2 \cos 2 \alpha-4 z \sigma^{-2}+a(r, \alpha),
$$
and
$$
B(z, r, \alpha) \triangleq 2 (z-1)^2 \sigma^{-1} r \sin \alpha+(z-1)^2 r^2 \sin 2 \alpha+b(r, \alpha).
$$

Now, for some small yet fixed $\eps > 0$, choose $N > 0$ large enough and then $r > 0$ small enough such that
\begin{itemize}
\item[(I)] for all $0 \leq z \leq N$ and all $\alpha \in [-\pi/2, 3\pi/2]$, $B(z, r, \alpha) \in (-\pi, \pi)$;
\item[(II)] for all $z \geq N$ and all $\alpha \in [-\pi/2+\eps, \pi/2-\eps] \cup [\pi/2+\eps, 3\pi/2-\eps]$,
$$
4z\sigma^{-2} \gg |a(r, \alpha)|, ~~~ |2 (z-1)^2 \sigma^{-1} r \cos \alpha| \gg |(z-1)^2 r^2 \cos 2 \alpha+a(r, \alpha)|
$$
and
$$
\left|\frac{\partial}{\partial \alpha} (2 (z-1)^2 \sigma^{-1} r \sin \alpha)\right| \gg \left|\frac{\partial}{\partial \alpha} ((z-1)^2 r^2 \sin 2 \alpha+b(r, \alpha))\right|.
$$
\end{itemize}

Note that for all $0 \leq z \leq N$, by (I), $\Phi_z(1/\omega)$ will not go around $-1$ (in any direction) for one complete round as $\alpha$ increases from $-\pi/2$ to $3\pi/2$. Next, we consider the case when $z \geq N$. Notice that, by (II), for any fixed $z \geq N$, as $\alpha$ increases from $-\pi/2+\eps$ to $\pi/2-\eps$, $B(z, r, \alpha)$ increases as well. If, for some $z \geq N$ and $\alpha_0 \in [-\pi/2, -\pi/2+\eps] \cup [\pi/2-\eps, \pi/2]$,
$$
A(z, r, \alpha_0) > 0,
$$
it then follows from (II) that there exists $\ell=\ell(\eps) > 0$ such that $\ell \to \infty$ as $\eps \to 0$ and for the same $z$ and any $\alpha \in [-\ell \eps, \ell \eps]$,
$$
A(z, r, \alpha) > 0.
$$
On the other hand, it follows from (II) that for any $z \geq N$ and for any $\alpha \in [\pi/2+\eps, 3 \pi/2-\eps]$,
$$
A(z, r, \alpha) < 0;
$$
straightforward computations also yield that for any $z \geq N$ and for any $\alpha \in [\pi/2, \pi/2+\eps] \cup [3 \pi/2-\eps, 3 \pi/2]$,
$$
A(z, r, \alpha) < 0.
$$
It then follows that $\Phi_z(1/\omega)$ will not go around $-1$ (in any direction) for one complete round as $\alpha$ increases from $\pi/2$ to $3\pi/2$.

We are now ready to conclude that as $\alpha$ increases from $-\pi/2$ to $3\pi/2$, for any $z \geq N$ with $\Phi_z(1/(\sigma^{-1} + r e^{i \alpha})) \neq -1$, $\Phi_z(1/\omega)$ will go around $-1$ anti-clockwise $k(z)$ times, where $k(z)$ is a non-negative integer; meanwhile, one checks that when $z$ is large enough, $k(z)$ is strictly positive. The idea can be roughly described as follows. Consider the ``trajectory'' of $\Phi_z(1/\omega)$ as $\alpha$ increases from $-\pi/2$ to $3\pi/2$. Obviously, $A(z, r, \alpha) > 0$ means the magnitude of the corresponding ``location'' is strictly bigger than $1$; $B(z, r, \alpha) > 0$ means at the corresponding ``location'', $\Phi_z(1/\omega)$ is going anti-clockwise. The above argument shows that given sufficiently small $\eps$ (and thus $\ell$ sufficiently large), for all the time when the ``location'' is at least $1$ away from the origin, ``more often'' $\Phi_z(1/\omega)$ goes around $-1$ anti-clockwise (for any $\alpha \in [-\pi/2, -\pi/2+\eps] \cup [\pi/2-\eps, \pi/2]$, $\Phi_z(1/\omega)$ may go around $-1$ clockwise, whereas for all $\alpha \in [-\ell \eps, \ell \eps]$, $\Phi_z(1/\omega)$ must go around $-1$  anti-clockwise).

So, for any analytic extension along the circle $\{\sigma^{-1}+r e^{i \alpha}: \alpha \in [-\pi/2, 3\pi/2]\}$ (from $\alpha=-\pi/2$ to $\alpha=3\pi/2$), we have proven that for any $z \geq 0$ with $\Phi_z(1/(\sigma^{-1} + r e^{i \alpha})) \neq -1$,
$$
\hspace{-1cm} \Im \left(\lim_{\alpha \to (3\pi/2)-} \log \left(1 + \Phi_z(1/(\sigma^{-1} + r e^{i \alpha}))
\right) \right)=\Im \left(\lim_{\alpha \to (-\pi/2)+} \log \left(1 + \Phi_z(1/(\sigma^{-1} + r e^{i \alpha})) \right) \right)+ 2k(z)\pi i,
$$
where $k(z)$ is a non-negative integer for all $z$ and a strictly positive integer for
all sufficiently large $z$. Using a similar argument, we can also prove that for any $z \leq 0$ with $\Phi_z^{-1}(1/(\sigma^{-1} + r e^{i \alpha})) \neq -1$,
$$
\hspace{-1cm} \Im \left(\lim_{\alpha \to (3\pi/2)-} \log \left(1 + \Phi_z^{-1}(1/(\sigma^{-1} + r e^{i \alpha}))
\right) \right)=\Im \left(\lim_{\alpha \to (-\pi/2)+} \log \left(1 + \Phi_z^{-1}(1/(\sigma^{-1} + r e^{i \alpha})) \right) \right)+ 2k(z)\pi i,
$$
where $k(z)$ is a non-negative integer for all $z$ and a strictly positive integer for
all sufficiently large $|z|$. This, however, implies that for the above analytic extension, (\ref{variant-of-the-second-term}) disagrees at $\alpha=-\pi/2$ and $\alpha=3\pi/2$, which is a
contradiction.

\end{exmp}

\section{A Complex Hilbert Metric} \label{Hilbert}

Recall that $\tilde{W}$ denote the complex version of $W$,
$$
\tilde{W}=\{w=(w_1, w_2, \cdots, w_{l}) \in \mathbb{C}^{l} : \sum_i
w_i=1\}.
$$
Let
$$
\tilde{W}^{+} = \{v \in \tilde{W}: \Re(v_i/v_j) >0 \mbox{ for all }
i, j\}.
$$
For $v,w \in \tilde{W}^{+}$, define
\begin{equation}  \label{ComplexHilbertMetric} \tilde{d}_H(v, w)=
\max_{i, j} \left| \log \left( \frac{w_i/w_j}{v_i/v_j} \right)
\right|,
\end{equation}
where $\log$ is taken as the principal branch of the complex
$\log(\cdot)$ function (i.e., the branch whose branch cut is the
negative real axis).  Since the principal branch of $\log$ is
additive on the right-half plane, $\tilde{d}_H$ is a metric on
$\tilde{W}^{+}$, which we call a {\em complex Hilbert metric} (for alternative complex Hilbert metrics, see~\cite{Rugh} and~\cite{Dubois}).

Let $M$ denote the set of all $l \times l$ stochastic matrices,
i.e.,
$$
M=\{\Pi=(\pi_{ij}) \in \mathbb{R}^{l \times l}: \pi_{ij} \geq 0, ~~
\sum_{j=1}^{l} \pi_{ij}=1\},
$$
and let $\tilde{M}$ denote the complex version of $M$, defined as
$$
\tilde{M}=\{\Pi=(\pi_{ij}) \in \mathbb{C}^{l \times l}:
\sum_{j=1}^{l} \pi_{ij} =1, \mbox{ for all } i\}.
$$
For a given positive $\Pi \in M$ and a small $\delta_1 > 0$, let
$\tilde{M}_{\Pi}(\delta_1)$  denote the $\delta_1$-neighborhood,
under the Euclidean metric, around $\Pi$ within $\tilde{M}$. For an
element $\tilde{\Pi} \in \tilde{M}_{\Pi}(\delta_1)$, similar to
(\ref{Induce}), $\tilde{\Pi}$ will induce a mapping
$f_{\tilde{\Pi}}$ on $\tilde{W}$. For a small $\delta_2 > 0$, let
$\tilde{W}_{W^{\circ}, H}(\delta_2)$ denote the
$\delta_2$-neighborhood of $W^{\circ}$ within $\tilde{W}^+$ under
the complex Hilbert metric, i.e.,
$$
\hspace{-0.3cm} \tilde{W}_{W^{\circ}, H}(\delta_2)=\{v=(v_1, v_2,
\cdots, v_{l}) \in \tilde{W}^+: \tilde{d}_H(v, u) \leq \delta_2,
\mbox{ for some } u \in W^{\circ}\}.
$$
The main result of~\cite{hm09a} states:
\begin{thm} \label{complex-contraction}
Let $\Pi$ be a positive matrix in $M$. For sufficiently small
$\delta_1, \delta_2 > 0$, there exists $0 < \rho_1 < 0$ such that
for any $\tilde{\Pi} \in \tilde{M}_{\Pi}(\delta_1)$,
$f_{\tilde{\Pi}}$ is a $\rho_1$-contraction mapping on
$\tilde{W}_{W^{\circ}, H}(\delta_2)$ under the complex Hilbert
metric in (\ref{ComplexHilbertMetric}).
\end{thm}

For $\delta > 0$, let $\mathbb{C}_{\mathbb{R}^+}[\delta]$  denote the ``$\delta$-cone'' of
$\mathbb{R}^+$ within $\mathbb{C}$, i.e.,
$$
\mathbb{C}_{\mathbb{R}^+}[\delta]=\{x+y i \in \mathbb{C}: x > 0,
-\delta x \leq y \leq \delta x\}.
$$
The following Lemma can be easily checked.
\begin{lem}  \label{max-of-two}
For sufficiently small $\delta_1 > 0$, there exists a positive
constant $L_1$ such that for any $\alpha, \beta \in
\mathbb{C}_{\mathbb{R}^+}[\delta_1]$
$$
|\log \alpha - \log \beta| \leq L_1 \max
\left(\frac{|\alpha-\beta|}{|\alpha|},
\frac{|\alpha-\beta|}{|\beta|} \right).
$$
\end{lem}
The following lemma essentially follows from the proof of Part $2$ of
Lemma 2.3 in~\cite{hm09a} (in particular, its Part $1$ is just a
rephrased version of Part $2$ of that lemma), allows
us to connect the complex Hilbert metric and the Euclidean metric.
We give a proof for completeness.
\begin{lem} \label{Brian}
\begin{enumerate}
\item For any $\delta > 0$, there exists $\xi > 0$ such that for any $\tilde{x} \in \tilde{W}^+$, $x \in W^{\circ}$ with $\tilde{d}_H(\tilde{x}, x) \leq \xi$, we have $\tilde{x}_i \in \tilde{W}_{W^{\circ}, H}(\delta)$ for all $i$.

\item For any $\zeta > 0$, there exists a constant $C > 0$ such that for any $\tilde{x}, \tilde{y} \in \tilde{W}^+$ with $|\tilde{x}-x|, |\tilde{y}-y| \leq \zeta$ for some $x, y \in W^{\circ}$, we have
    $$
    |\tilde{x}-\tilde{y}|\leq C \tilde{d}_H(\tilde{x}, \tilde{y}).
    $$
\end{enumerate}
\end{lem}

\begin{proof}

We only prove Part $2$. Let $\xi = \tilde{d}_H(\tilde{x},
\tilde{y})$. Then we have for all
$i, j$,
$$
\left| \log \left(
\frac{\tilde{x}_i/\tilde{y}_i}{\tilde{x}_j/\tilde{y}_j} \right)
\right| \leq \xi.
$$
There exists $C_1 >0$ such that for $\xi$ sufficiently small, and
for all $i, j$, $\left|
\frac{\tilde{x}_i/\tilde{y}_i}{\tilde{x}_j/\tilde{y}_j} -1 \right|
\leq C_1 \xi$. Let $\alpha_j = \tilde{x}_j/\tilde{y}_j$.  Then for
all $i,j$,
$$
|\tilde{x}_i - \alpha_j \tilde{y}_i|  \le C_1 \xi |\alpha_j|
|\tilde{y}_i|,
$$
and so
$$
|1 - \alpha_j|= \left| \sum_{i=1}^n (\tilde{x}_i - \alpha_j
\tilde{y}_i) \right| \le \sum_{i=1}^n |\tilde{x}_i - \alpha_j
\tilde{y}_i| \le C_1 \xi |\alpha_j| \sum_{i=1}^n |\tilde{y}_i| = C_1
(1+ B \zeta) \xi |\alpha_j|.
$$
It follows that $|\tilde{x}_j - \tilde{y}_j| \le C_1 (1+ B \zeta)
\xi |\tilde{x}_j| \leq C_1 (1+ B \zeta) \xi (x_j + \zeta)$, which
implies Part $2$, if $\xi$ is sufficiently small.
\end{proof}

\section{Proof of Theorem~\ref{main}} \label{III}

The following lemma is an analog of Lemma~\ref{Cauchy-Uniformity}.

\begin{lem} \label{Uniformity}
For any $\delta > 0$, there exist $r_1, r_2 > 0$ such that for any $(\vec{\eps}, \vec{\theta}) \in \mathbb{C}_{\vec{\eps}_0}^{m_1}(r_1) \times \mathbb{C}_{\vec{\theta}_0}^{m_2}(r_2)$, any $z \in \mathcal{Z}$ and any $x \in W$, we have
$$
\tilde{d}_H(f^{\vec{\eps}, \vec{\theta}}_{z}(x), f^{\vec{\eps}_0, \vec{\theta}_0}_{z}(x)) \leq \delta.
$$
\end{lem}

\begin{proof}

Since all $\pi_{ij}(\vec{\eps}_0)$ are strictly positive, for any $\delta_1 > 0$, there exists $r_1 > 0$ such that for all $i, j$ and all $\vec{\eps} \in \mathbb{C}_{\vec{\eps}_0}^{m_1}(r_1)$,
$$
\frac{|\pi_{ij}^{\vec{\eps}}-\pi_{ij}(\vec{\eps}_0)|}{\pi_{ij}(\vec{\eps}_0)}
\leq \delta_1.
$$
Now, for any $x=(x_1, x_2, \cdots, x_{l}) \in W$, any $j$ and any $\vec{\eps} \in \mathbb{C}_{\vec{\eps}_0}^{m_1}(r_1)$, we have
$$
\left| \frac{\sum_{i=1}^{l}x_i
(\pi_{ij}^{\vec{\eps}}-\pi_{ij}(\vec{\eps}_0))}{\sum_{i=1}^{l}x_i
\pi_{ij}(\vec{\eps}_0)} \right| =\left| \frac{\sum_{i=1}^{l}  x_i
\pi_{ij}(\vec{\eps}_0)
(\pi_{ij}^{\vec{\eps}}-\pi_{ij}(\vec{\eps}_0))/\pi_{ij}(\vec{\eps}_0)}{\sum_{i=1}^{l}x_i
\pi_{ij}(\vec{\eps}_0)} \right| \leq \delta_1.
$$
Thus, for any $\delta_2 > 0$, choosing $\delta_1$ sufficiently small, we have
$$
\left| \log \frac{\sum_{i=1}^{l}x_i \pi_{ij}^{\vec{\eps}}}{
\sum_{i=1}^{l}x_i \pi_{ij}(\vec{\eps}_0)} \right| =\left| \log
\left(1+ \frac{\sum_{i=1}^{l}x_i
(\pi_{ij}^{\vec{\eps}}-\pi_{ij}(\vec{\eps}_0)) }{\sum_{i=1}^{l}x_i
\pi_{ij}(\vec{\eps}_0)} \right) \right| \leq \delta_2.
$$
Notice that
$$
\tilde{d}_H(f^{\vec{\eps}, \vec{\theta}}_{z}(x), f^{\vec{\eps}_0,
\vec{\theta}_0}_{z}(x))=\max_{j, k} \left| \log
\frac{\sum_{i=1}^{l}x_i \pi_{ij}^{\vec{\eps}}
q^{\vec{\theta}}(z|j)}{\sum_{i=1}^{l}x_i \pi_{ij}(\vec{\eps}_0)
q^{\vec{\theta}_0}(z|j)}-\log \frac{\sum_{i=1}^{l}x_i
\pi_{ik}^{\vec{\eps}} q^{\vec{\theta}}(z|k)}{\sum_{i=1}^{l}x_i
\pi_{ik}(\vec{\eps}_0) q^{\vec{\theta}_0}(z|k)} \right|
$$
$$
=\max_{j, k} \left| \log \frac{\sum_{i=1}^{l}x_i
\pi_{ij}^{\vec{\eps}}}{\sum_{i=1}^{l}x_i
\pi_{ij}(\vec{\eps}_0)}+\log
\frac{q^{\vec{\theta}}(z|j)}{q^{\vec{\theta}_0}(z|j)}-\log
\frac{\sum_{i=1}^{l}x_i \pi_{ik}^{\vec{\eps}}}{\sum_{i=1}^{l}x_i
\pi_{ik}(\vec{\eps}_0)}-\log
\frac{q^{\vec{\theta}}(z|k)}{q^{\vec{\theta}_0}(z|k)} \right|.
$$
It then follows from the second inequality of (\ref{dominated-by-real}) that for any $\delta > 0$, there exist $r_1, r_2 > 0$ such that for any $(\vec{\eps}, \vec{\theta}) \in \mathbb{C}_{\vec{\eps}_0}^{m_1}(r_1) \times \mathbb{C}_{\vec{\theta}_0}^{m_2}(r_2)$ and any $x \in W$, we have
$$
\tilde{d}_H(f^{\vec{\eps}, \vec{\theta}}_{z}(x), f^{\vec{\eps}_0, \vec{\theta}_0}_{z}(x)) \leq \delta.
$$

\end{proof}

The following lemma, roughly speaking, says that when we perturb $(\vec{\eps}_0, \vec{\theta}_0)$ ``a bit'' to $(\vec{\eps}, \vec{\theta})$, $f^{\vec{\eps}, \vec{\theta}}_{z}$ is still a contraction mapping on a complex neighborhood of $W^{\circ}$, and the contraction coefficient is uniform over all the values of $z$ and $\vec{\theta}$. More precisely, recalling $\tilde{W}_{W^{\circ}, H}(\delta)$ denote the $\delta$-neighborhood of $W^{\circ}$ of $\tilde{W}$ under the complex Hilbert metric, we have
\begin{lem} \label{z-complex-contraction}
For sufficiently small $r_1, r_2, \delta > 0$, there exists $0 < \rho_1 < 1$ such that for any $(\vec{\eps}, \vec{\theta}) \in \mathbb{C}_{\vec{\eps}_0}^{m_1}(r_1) \times \mathbb{C}_{\vec{\theta}_0}^{m_2}(r_2)$ and any $z \in \mathcal{Z}$, $f^{\vec{\eps}, \vec{\theta}}_{z}$ is a $\rho_1$-contraction mapping on $\tilde{W}_{W^{\circ}, H}(\delta)$ under the complex Hilbert metric in (\ref{ComplexHilbertMetric}).
\end{lem}

\begin{proof}
By (\ref{one-column-dominant}), we can choose $r_1, r_2, \delta > 0$ sufficiently small such that for any $z \in \mathcal{Z}$, any $(\vec{\eps},\vec{\theta}) \in \mathbb{C}_{\vec{\eps}_0}^{m_1}(r_1) \times \mathbb{C}_{\vec{\theta}_0}^{m_2}(r_2)$ and any $u, v \in \tilde{W}_{W^{\circ}, H}(\delta)$,
$$
\tilde{d}_H(u \Pi^{\vec{\eps},\vec{\theta}}(z), v \Pi^{\vec{\eps},\vec{\theta}}(z))
$$
is well-defined. Moreover, it can be easily checked that
\begin{equation}  \label{key-observation}
\tilde{d}_H(u \Pi^{\vec{\eps},\vec{\theta}}(z), v \Pi^{\vec{\eps},\vec{\theta}}(z))=\tilde{d}_H(u \Pi^{\vec{\eps}}, v \Pi^{\vec{\eps}}).
\end{equation}
The lemma then immediately follows from Theorem~\ref{complex-contraction}.
\end{proof}

The following lemma is also needed.

\begin{lem}  \label{a-b}
\begin{enumerate}
\item For any $\delta > 0$, there exist $r_1, r_2 > 0$ such that for any $(\vec{\eps}, \vec{\theta}) \in \mathbb{C}_{\vec{\eps}_0}^{m_1}(r_1) \times \mathbb{C}_{\vec{\theta}_0}^{m_2}(r_2)$ and for any $z_{-n}^0 \in \mathcal{Z}^{n+1}$ and $-n-1 \leq i \leq -1$,
\begin{equation} \label{confined-orbit}
x_i^{\vec{\eps}, \vec{\theta}}(z_{-n}^i) \in \tilde{W}_{W^{\circ}, H}(\delta),
\end{equation}
and
\begin{equation} \label{stay-in-cone}
p^{\vec{\eps}, \vec{\theta}}(z_0|z_{-n}^{-1}) \in \mathbb{C}_{\mathbb{R}^+}[\delta].
\end{equation}

\item There exist $r_1, r_2 > 0$ such that for all $z_{-n}^0 \in \mathcal{Z}^{n+1}$, $p^{\vec{\eps}, \vec{\theta}}(z_0|z_{-n}^{-1})$ is analytic on $\mathbb{C}_{\vec{\eps}_0}^{m_1}(r_1) \times \mathbb{C}_{\vec{\theta}_0}^{m_2}(r_2)$.

\item For sufficiently small $r_1, r_2 > 0$, there exist $0 < \rho_1 < 1$ and a positive constant $L_1$ such that for any two $\mathcal{Z}$-valued sequences $\{a_{-n_1}^0\}$ and $\{b_{-n_2}^0\}$ with $a_{-n}^0=b_{-n}^0$ and for all $(\vec{\eps}, \vec{\theta}) \in \mathbb{C}_{\vec{\eps}_0}^{m_1}(r_1) \times \mathbb{C}_{\vec{\theta}_0}^{m_2}(r_2)$, we have
$$
|p^{\vec{\eps}, \vec{\theta}}(a_0|a_{-n_1}^{-1})-p^{\vec{\eps}, \vec{\theta}}(b_0|b_{-n_2}^{-1})| \leq L_1 \rho_1^n \sup_{(y', \vec{\theta}') \in \mathcal{Y} \times \mathbb{C}_{\vec{\theta}_0}^{m_2}(r_2)} |q^{\vec{\theta}'}(a_0|y')|.
$$
\end{enumerate}
\end{lem}

\begin{proof}

1. By Lemma~\ref{z-complex-contraction}, we can choose $r_1, r_2, \delta > 0$ sufficiently small such that there exists $0 < \rho_1 < 1$ such that for all $(\vec{\eps}, \vec{\theta}) \in \mathbb{C}_{\vec{\eps}_0}^{m_1}(r_1) \times \mathbb{C}_{\vec{\theta}_0}^{m_2}(r_2)$, $f_z^{\vec{\eps}, \vec{\theta}}$ is a $\rho_1$-contraction mapping on $\tilde{W}_{W^{\circ}, H}(\delta)$ under the complex Hilbert metric.

Now, choose $r_1, r_2 > 0$ so small (the existence of $r_1, r_2$ is guaranteed by Lemma~\ref{Uniformity}) such that for any $z \in \mathcal{Z}$, for all $x \in W$, all $(\vec{\eps}, \vec{\theta}) \in \mathbb{C}_{\vec{\eps}_0}^{m_1}(r_1) \times \mathbb{C}_{\vec{\theta}_0}^{m_2}(r_2)$
\begin{equation} \label{bounded-1}
\hspace{-1cm} \tilde{d}_H(f^{\vec{\eps}, \vec{\theta}}_{z}(x), f^{\vec{\eps}_0, \vec{\theta}_0}_{z}(x)) \leq \delta (1-\rho_1),
\end{equation}
and for all $\vec{\eps} \in \mathbb{C}_{\vec{\eps}_0}^{m_1}(r_1)$,
\begin{equation} \label{bounded-2}
\tilde{d}_H(\pi^{\vec{\eps}}, \pi(\vec{\eps}_0)) \leq \delta (1-\rho_1).
\end{equation}
We then deduce that
\begin{align}  \label{star}
\nonumber \tilde{d}_H(x^{\vec{\eps}, \vec{\theta}}_{i+1}, x^{\vec{\eps}_0, \vec{\theta}_0}_{i+1}) &=\tilde{d}_H(f^{\vec{\eps}, \vec{\theta}}_{z_{i+1}}(x^{\vec{\eps}, \vec{\theta}}_i), f^{\vec{\eps}_0, \vec{\theta}_0}_{z_{i+1}}(x^{\vec{\eps}_0, \vec{\theta}_0}_i)) \\
& \leq \tilde{d}_H(f^{\vec{\eps}, \vec{\theta}}_{z_{i+1}}(x^{\vec{\eps}, \vec{\theta}}_i), f^{\vec{\eps}, \vec{\theta}}_{z_{i+1}}(x^{\vec{\eps}_0, \vec{\theta}_0}_i))+\tilde{d}_H(f^{\vec{\eps}, \vec{\theta}}_{z_{i+1}}(x^{\vec{\eps}_0, \vec{\theta}_0}_i), f^{\vec{\eps}_0, \vec{\theta}_0}_{z_{i+1}}(x^{\vec{\eps}_0, \vec{\theta}_0}_i)).
\end{align}
Then, by (\ref{bounded-1}), (\ref{bounded-2}) and (\ref{star}), for $i > -n-1$, we have
$$
\tilde{d}_H(x^{\vec{\eps}, \vec{\theta}}_{i+1}, x^{\vec{\eps}_0, \vec{\theta}_0}_{i+1}) \leq \rho \tilde{d}_H(x^{\vec{\eps}, \vec{\theta}}_i, x^{\vec{\eps}_0, \vec{\theta}_0}_i)+ \delta (1-\rho_1).
$$
So, for all $i$,
$$
\tilde{d}_H(x^{\vec{\eps}, \vec{\theta}}_{i+1}, x^{\vec{\eps}_0, \vec{\theta}_0}_{i+1}) \leq \delta,
$$
and thus for all $i$, we have $x^{\vec{\eps}, \vec{\theta}}_{i+1} \in \tilde{W}_{W^{\circ}, H}(\delta)$, as desired. This, together with (\ref{amalg}) and Lemma~\ref{Brian} (Part $2$), implies (\ref{stay-in-cone}).

2. It follows from (\ref{dominated-by-real}(i)) and Part $1$ of Lemma~\ref{Brian} that for sufficiently small $r_1, r_2, \delta > 0$ and any $z \in \mathcal{Z}$, $f_z^{\vec{\eps}, \vec{\theta}}(x)$ is analytic with respect to $(\vec{\eps}, \vec{\theta}, x) \in \mathbb{C}_{\vec{\eps}_0}^{m_1}(r_1) \times \mathbb{C}_{\vec{\theta}_0}^{m_2}(r_2) \times \tilde{W}_{W^{\circ}, H}(\delta)$. It follows from this fact, the iterative nature of $x^{\vec{\eps}, \vec{\theta}}_i$ (see (\ref{iter0})) and Part $1$ that
for sufficiently small $r_1, r_2 > 0$, each $x^{\vec{\eps}, \vec{\theta}}_i$ is analytic
on $\mathbb{C}_{\vec{\eps}_0}^{m_1}(r_1) \times \mathbb{C}_{\vec{\theta}_0}^{m_2}(r_2)$. Part $2$ then immediately follows from (\ref{amalg}).

3. For all $(\vec{\eps}, \vec{\theta}) \in \mathbb{C}_{\vec{\eps}_0}^{m_1}(r_1) \times \mathbb{C}_{\vec{\theta}_0}^{m_2}(r_2)$, we write
$$
x^{\vec{\eps}, \vec{\theta}}_{i, a}=x^{\vec{\eps}, \vec{\theta}}_i(a_{-n_1}^i)=p^{\vec{\eps}, \vec{\theta}}(y_i=\cdot \;|a_{-n_1}^i),
$$
$$
x^{\vec{\eps}, \vec{\theta}}_{i,b}=x^{\vec{\eps}, \vec{\theta}}_i(b_{-n_2}^i)=p^{\vec{\eps}, \vec{\theta}}(y_i=\cdot \;|b_{-n_2}^i).
$$
Apparently we have
$$
x^{\vec{\eps}, \vec{\theta}}_{i+1, a}=f^{\vec{\eps}, \vec{\theta}}_{a_{i+1}}(x^{\vec{\eps}, \vec{\theta}}_{i, a}),
\qquad
x^{\vec{\eps}, \vec{\theta}}_{i+1, b}=f^{\vec{\eps}, \vec{\theta}}_{b_{i+1}}(x^{\vec{\eps}, \vec{\theta}}_{i, b}).
$$
Note that there exists a positive constant $L'_1$ such that
$$
\tilde{d}_H(x^{\vec{\eps}, \vec{\theta}}_{-n, a}, x^{\vec{\eps}, \vec{\theta}}_{-n, b}) \leq L'_1,
$$
for all $(\vec{\eps}, \vec{\theta}) \in \mathbb{C}_{\vec{\eps}_0}^{m_1}(r_1) \times \mathbb{C}_{\vec{\theta}_0}^{m_2}(r_2)$, where $r_1, r_2 >0$ are chosen sufficiently small.
Then from (\ref{confined-orbit}), we have
$$
\tilde{d}_H(x^{\vec{\eps}, \vec{\theta}}_{-1, a}, x^{\vec{\eps}, \vec{\theta}}_{-1, b}) \leq L'_1 \rho_1^{n-1}.
$$
Therefore, by Lemma~\ref{Brian}, there exists a positive constant $L''_1$ independent of $n_1, n_2$ such that for any $(\vec{\eps}, \vec{\theta}) \in \mathbb{C}_{\vec{\eps}_0}^{m_1}(r_1) \times \mathbb{C}_{\vec{\theta}_0}^{m_2}(r_2)$, we have
\begin{equation}   \label{F-rho}
|x^{\vec{\eps}, \vec{\theta}}_{-1, a}-x^{\vec{\eps}, \vec{\theta}}_{-1, b}| \leq L''_1 \rho_1^n.
\end{equation}
Now, using (\ref{amalg}) and the fact that
$$
p^{\vec{\eps}, \vec{\theta}}(a_0)=\sum_y \pi_y^{\vec{\eps}}
q^{\vec{\theta}}(a_0|y),
$$
we conclude that there is a positive constant $L_1$, independent of $n_1, n_2$ such that
\begin{equation}   \label{F-rho'}
|p^{\vec{\eps}, \vec{\theta}}(a_0|a_{-n_1}^{-1})-p^{\vec{\eps}, \vec{\theta}}(b_0|b_{-n_2}^{-1})| \leq L_1 \rho_1^n \sup_{(y', \vec{\theta}') \in \mathcal{Y} \times \mathbb{C}_{\vec{\theta}_0}^{m_2}(r_2)} |q^{\vec{\theta}'}(a_0|y')|.
\end{equation}
We then have finished the proof.
\end{proof}

We are now ready for the proof of Theorem~\ref{main}.

\begin{proof}[Proof of Theorem~\ref{main}]

We first prove that there exist $r_1, r_2 > 0$ such that for any $n$, $H^{\vec{\eps}, \vec{\theta}}_n(Z)$ is analytic on $\mathbb{C}_{\vec{\eps}_0}^{m_1}(r_1) \times \mathbb{C}_{\vec{\theta}_0}^{m_2}(r_2)$.

For a fixed $n$, recall that
$$
H_n^{\vec{\eps}, \vec{\theta}}(Z)=-\int_{\mathcal{Z}^{n+1}} p^{\vec{\eps}, \vec{\theta}}(z_{-n}^0) \log p^{\vec{\eps}, \vec{\theta}}(z_0|z_{-n}^{-1}) dz_{-n}^0,
$$
where
$$
p^{\vec{\eps}, \vec{\theta}}(z_{-n}^0)=\sum_{y_{-n}^0} p^{\vec{\eps}}(y_{-n}^0) \prod_{i=-n}^0 q^{\vec{\theta}}(z_i|y_i)
$$
and
$$
p^{\vec{\eps}, \vec{\theta}}(z_0|z_{-n}^{-1}) = x_{-1}^{\vec{\eps}, \vec{\theta}}(z_{-n}^{-1}) \Pi^{\vec{\eps}, \vec{\theta}}(z_0) \mathbf{1}.
$$
Now, for any $(\vec{\eps}, \vec{\theta}) \in \mathbb{C}_{\vec{\eps}_0}^{m_1}(r_1) \times \mathbb{C}_{\vec{\theta}_0}^{m_2}(r_2)$, we have
\begin{equation} \label{main-triangle}
|p^{\vec{\eps}, \vec{\theta}}(z_{-n}^0)| \leq \sum_{y_{-n}^0} |p^{\vec{\eps}}(y_{-n}^0)| \prod_{i=-n}^0 |q^{\vec{\theta}}(z_i|y_i)| \leq \sum_{y_{-n}^0} |p^{\vec{\eps}}(y_{-n}^0)| \prod_{i=-n}^0 \sup_{(y', \vec{\theta}') \in \mathcal{Y} \times \mathbb{C}_{\vec{\theta}_0}^{m_2}(r_2)}| q^{\vec{\theta}'}(z_i|y')|.
\end{equation}
And, by (\ref{confined-orbit}), for sufficiently small $r_1, r_2 > 0$, there exist $C_1, C_2 > 0$ such that
\begin{equation} \label{main-star}
C_1 \inf_{(y', \vec{\theta}') \in \mathcal{Y} \times \mathbb{C}_{\vec{\theta}_0}^{m_2}(r_2)} |q^{\vec{\theta}'}(z_0|y')| \leq |p^{\vec{\eps}, \vec{\theta}}(z_0|z_{-n}^{-1})| \leq C_2 \sup_{(y', \vec{\theta}') \in \mathcal{Y} \times \mathbb{C}_{\vec{\theta}_0}^{m_2}(r_2)} |q^{\vec{\theta}'}(z_0|y')|,
\end{equation}
which, together with (\ref{stay-in-cone}), implies that for some $C_3 > 0$,
$$
|\log p^{\vec{\eps}, \vec{\theta}}(z_0|z_{-n}^{-1})| \leq C_3+ \max \{|\log \sup_{(y', \vec{\theta}') \in \mathcal{Y} \times \mathbb{C}_{\vec{\theta}_0}^{m_2}(r_2)} |q^{\vec{\theta}'}(z_0|y')|, \log \inf_{(y', \vec{\theta}') \in \mathcal{Y} \times \mathbb{C}_{\vec{\theta}_0}^{m_2}(r_2)} |q^{\vec{\theta}'}(z_0|y')|\}.
$$
This, together with (\ref{max-of-three}), implies that on $\mathbb{C}_{\vec{\eps}_0}^{m_1}(r_1) \times \mathbb{C}_{\vec{\theta}_0}^{m_2}(r_2)$,
\begin{equation} \label{main-absolute-value-bounded}
\int_{\mathcal{Z}^{n+1}} \sup_{(\vec{\eps}, \vec{\theta}) \in \mathbb{C}_{\vec{\eps}_0}^{m_1}(r_1) \times \mathbb{C}_{\vec{\theta}_0}^{m_2}(r_2)} \left| p^{\vec{\eps}, \vec{\theta}}(z_{-n}^0) \log p^{\vec{\eps}, \vec{\theta}}(z_0|z_{-n}^{-1}) \right| dz_{-n}^0 < \infty
\end{equation}
By Lemma~\ref{exercise} (Part $2$), $H_n^{\vec{\eps}, \vec{\theta}}(Z_0|Z_{-n}^{-1})$
is analytic on $\mathbb{C}_{\vec{\eps}_0}^{m_1}(r_1) \times \mathbb{C}_{\vec{\theta}_0}^{m_2}(r_2)$.

Now, to prove the theorem, we only need to prove that there exist $r_1, r_2 > 0$ such that the $H^{\vec{\eps}, \vec{\theta}}_n(Z)$, as $n \to \infty$, uniformly converges on $\mathbb{C}_{\vec{\eps}_0}^{m_1}(r_1) \times \mathbb{C}_{\vec{\theta}_0}^{m_2}(r_2)$. Note that
\begin{align*}
\hspace{-1cm} |H^{\vec{\eps}, \vec{\theta}}_{n+1}(Z)-H^{\vec{\eps}, \vec{\theta}}_n(Z)| &=\left| \int_{\mathcal{Z}^{n+2}} p^{\vec{\eps}, \vec{\theta}}(z_{-n-1}^{0}) \log p^{\vec{\eps}, \vec{\theta}}(z_0|z_{-n-1}^{-1}) dz_{-n-1}^0 - \int_{\mathcal{Z}^{n+1}} p^{\vec{\eps}, \vec{\theta}}(z_{-n}^{0}) \log p^{\vec{\eps}, \vec{\theta}}(z_0|z_{-n}^{-1}) dz_{-n}^0 \right| \\
&=\left| \int_{\mathcal{Z}^{n+2}} p^{\vec{\eps}, \vec{\theta}}(z_{-n-1}^{0}) (\log p^{\vec{\eps}, \vec{\theta}}(z_0|z_{-n-1}^{-1})-\log p^{\vec{\eps}, \vec{\theta}}(z_0|z_{-n}^{-1})) dz_{-n-1}^0 \right|.
\end{align*}
Fix $(\vec{\eps}, \vec{\theta}) \in \mathbb{C}_{\vec{\eps}_0}^{m_1}(r_1) \times \mathbb{C}_{\vec{\theta}_0}^{m_2}(r_2)$. Then, by Lemmas~\ref{max-of-two} and~\ref{a-b}, either we have, for some $0 < \rho_1 < 1$, $L'_1 > 0$ and some $\delta_1$ with $(1+\delta_1) \rho_1 < 1$
$$
| p^{\vec{\eps}, \vec{\theta}}(z_{-n-1}^{0}) (\log p^{\vec{\eps}, \vec{\theta}}(z_0|z_{-n-1}^{-1})-\log p^{\vec{\eps}, \vec{\theta}}(z_0|z_{-n}^{-1})) | \leq L'_1 \left| p^{\vec{\eps}, \vec{\theta}}(z_{-n-1}^{0}) \frac{p^{\vec{\eps}, \vec{\theta}}(z_0|z_{-n-1}^{-1})-p^{\vec{\eps}, \vec{\theta}}(z_0|z_{-n}^{-1})}{p^{\vec{\eps}, \vec{\theta}}(z_0|z_{-n-1}^{-1})} \right|
$$
$$
\hspace{-1cm} \leq L'_1 |p^{\vec{\eps}, \vec{\theta}}(z_{-n-1}^{-1})| L_1 \rho_1^n \sup_{(y', \vec{\theta}') \in \mathcal{Y} \times \mathbb{C}_{\vec{\theta}_0}^{m_2}(r_2)} |q^{\vec{\theta}'}(z_0|y')|,
$$
or we have, for some $0 < \rho_1 < 1$, $L'_1 > 0$ and some $\delta_1$ with $(1+\delta_1) \rho_1 < 1$,
$$
|p^{\vec{\eps}, \vec{\theta}}(z_{-n-1}^{0}) (\log p^{\vec{\eps}, \vec{\theta}}(z_0|z_{-n-1}^{-1})-\log p^{\vec{\eps}, \vec{\theta}}(z_0|z_{-n}^{-1})) |
\leq L'_1 \left|p^{\vec{\eps}, \vec{\theta}}(z_{-n-1}^{0}) \frac{p^{\vec{\eps}, \vec{\theta}}(z_0|z_{-n-1}^{-1})-p^{\vec{\eps}, \vec{\theta}}(z_0|z_{-n}^{-1}))}{p^{\vec{\eps}, \vec{\theta}}(z_0|z_{-n}^{-1})} \right|
$$
$$
\leq L'_1 |p^{\vec{\eps}, \vec{\theta}}(z_{-n}^{-1})| |p^{\vec{\eps}, \vec{\theta}}(z_{-n-1}|z_{-n}^{0}) |  L_1 \rho_1^n \sup_{(y', \vec{\theta}') \in \mathcal{Y} \times \mathbb{C}_{\vec{\theta}_0}^{m_2}(r_2)} |q^{\vec{\theta}'}(z_0|y')|.
$$
Notice that for any given $\delta >0$, there exist $r_1, r_2 > 0$ such that for all $\vec{\eps} \in \mathbb{C}_{\vec{\eps}_0}^m(r_1)$,
$$
|\pi_{y_{-n}}^{\vec{\eps}}| \leq (1+\delta)
\pi_{y_{-n}}^{\vec{\eps}_0}, \qquad |\pi_{y_i y_{i+1}}^{\vec{\eps}}|
\leq (1+\delta) \pi_{y_i y_{i+1}}^{\vec{\eps}_0},
$$
and for any $y \in \mathcal{Y}$ and all $\vec{\theta} \in \mathbb{C}_{\vec{\theta}_0}^{m_2}(r_2)$,
$$
\int_{\mathcal{Z}} |q^{\vec{\theta}}(z|y)| dz \leq (1+\delta) \int_{\mathcal{Z}} q^{\vec{\theta}_0}(z|y) dz=1+\delta,
$$
(here we have used the fact that $\int_{\mathcal{Z}} |q^{\vec{\theta}}(z|y)| dz$ is a continuous function of $\vec{\theta} \in \mathbb{C}_{\vec{\theta}_0}^{m_2}(r_2)$; this follows from Lemma~\ref{exercise} (Part $1$)). It then follows from (\ref{main-triangle}) that
$$
\int_{\mathcal{Z}^{n}} |p^{\vec{\eps}, \vec{\theta}}(z_{-n}^{-1})| dz_{-n}^{-1} \leq (1+\delta)^{2n}, \qquad \int_{\mathcal{Z}^{n+1} }|p^{\vec{\eps}, \vec{\theta}}(z_{-n-1}^{-1})| dz_{-n-1}^{-1} \leq (1+\delta)^{2(n+1)}.
$$
Moreover, similar to (\ref{main-star}), we have for some $C_4, C_5 > 0$,
$$
C_4 \inf_{(y', \vec{\theta}') \in \mathcal{Y} \times \mathbb{C}_{\vec{\theta}_0}^{m_2}(r_2)} |q^{\vec{\theta}'}(z_{-n-1}|y')| \leq |p^{\vec{\eps}, \vec{\theta}}(z_{-n-1}|z_{-n}^0)| \leq C_5 \sup_{(y', \vec{\theta}') \in \mathcal{Y} \times \mathbb{C}_{\vec{\theta}_0}^{m_2}(r_2)} |q^{\vec{\theta}'}(z_{-n-1}|y')|.
$$

By choosing $\delta > 0$ sufficiently small, we can combine all the relevant inequalities above to obtain some $L > 0$ and some $0 < \rho < 1$ such that for all $(\vec{\eps}, \vec{\theta}) \in \mathbb{C}_{\vec{\eps}_0}^{m_1}(r_1) \times \mathbb{C}_{\vec{\theta}_0}^{m_2}(r_2)$,
$$
|H^{\vec{\eps}, \vec{\theta}}_{n+1}(Z)-H^{\vec{\eps}, \vec{\theta}}_n(Z)|
\leq \int_{\mathcal{Z}^{n+2}} | p^{\vec{\eps}, \vec{\theta}}(z_{-n-1}^{0}) (\log p^{\vec{\eps}, \vec{\theta}}(z_0|z_{-n-1}^{-1})-\log p^{\vec{\eps}, \vec{\theta}}(z_0|z_{-n}^{-1})) | dz_{-n-1}^0
\leq L \rho^n,
$$
which implies the analyticity of $H^{\vec{\eps}, \vec{\theta}}(Z)$ around $(\vec{\eps}_0, \vec{\theta}_0)$.

\end{proof}

\end{document}